\documentclass[11pt]{amsart}

\usepackage{mystyle}

\begin{document}

\title{Improved bounds for the chromatic index of $k$-uniform hypergraphs}

\date{July 3, 2026}

\author{Sarah Frederickson} 
\thanks{Frederickson is supported by the Department of Education Graduate Assistance in Areas of National Need (GAANN) program at the Georgia Institute of Technology (award \#P200A240169)}
\address{Georgia Institute of Technology}
\email{sfrederickson3@gatech.edu}

\author{Yanli Hao}
\address{Georgia Institute of Technology}
\email{yhao98@gatech.edu}

\author{Tom Kelly}
\thanks{Kelly's research is supported by the National Science Foundation under Grant No. DMS-2247078.}
\address{Georgia Institute of Technology}
\email{tom.kelly@gatech.edu}

\begin{abstract}
  In 1997, Alon and Kim conjectured that if $\cH$ is a $k$-uniform $t$-simple hypergraph with maximum degree $D$ sufficiently large, then the chromatic index $\chi'(\cH)$ is upper bounded by $(t-1+1/t+\eps)D$. Using probabilistic techniques and a nibble coloring method, we prove a general coloring theorem stating that a $k$-uniform $t$-simple hypergraph $\cH$ with large maximum degree $D$ satisfies $$\chi'(\cH) \le (b+\eps)kD,$$ where $b$ is a particular parameter derived from local structural information about $\cH$. We use structural techniques to prove sharp upper bounds on $b$ in the 3-uniform 2-simple, and 3-uniform 3-simple cases. In particular, we deduce as a corollary that for sufficiently large $D$, every $3$-uniform $2$-simple and $3$-simple hypergraph of maximum degree at most $D$ has chromatic index at most $2.3581D$ and $2.6791D$, respectively.  
\end{abstract}

\maketitle

\section{Introduction}

A \textit{(multi)hypergraph} $\cH$ consists of a vertex set $V(\cH)$ and a (multi)set $E(\cH)$ of (hyper)edges $e$ with $e\subseteq V(\cH)$. A hypergraph is \textit{$k$-uniform} if every edge has size $k$, and \textit{$t$-simple} if every two distinct edges intersect in at most $t$ vertices.

The chromatic index $\chi'(\cH)$ of a hypergraph $\cH$ is the smallest number of colors needed to color the edges of $\cH$ so that intersecting edges receive distinct colors. In 1997, Alon and Kim \cite{AK97} posed the following conjecture.

\begin{conjecture}[Alon and Kim \cite{AK97}]\label{Alon Kim}
    For every $k\ge t\ge 1$ and $\eps>0$, there exists $D_0$ such that the following holds. For every $D\ge D_0$, if $\mathcal{H}$ is a $k$-uniform, $t$-simple hypergraph with maximum degree at most $D$, then $$   \chi'(\mathcal{H})\le (t-1+1/t + \eps)D.$$
\end{conjecture}

Alon and Kim showed that this conjecture, if true, would be asymptotically tight for each $k\ge t$ for which there is a projective plane of order $t-1$.

The first nontrivial case is where $k=2$, that is, when $\cH$ is a graph, which is resolved by Vizing's Theorem \cite{Vi65} for $t=1$, and by Shannon's Theorem \cite{Sh49} for $t=2$. For the case of $t=1$ and $k$ arbitrary, this is resolved by a result of Pippenger and Spencer \cite{PS89}.

Alon and Kim's main result in \cite{AK97} is that an \textit{intersecting} $k$-uniform $t$-simple hypergraph of large maximum degree $D$ has at most $(t-1+1/t)D$ edges. Thus, this expression is a bound on the clique number of the line graph of a $k$-uniform $t$-simple hypergraph, but the chromatic number of a graph can be quite far from the clique number in general.

A simple greedy coloring bound implies that since the line graph of a $k$-uniform hypergraph has maximum degree at most $k(D-1)$, the chromatic index of the hypergraph $\cH$ satisfies $\chi'(\cH) \le k(D-1)+1$. To our knowledge, 
and as recently noted by Boyadzhiyska, Lang, Lo, and Molloy \cite{BLLMsimultaneous} (for the $t=k$ case),
there is no nontrivial published upper bound improving this.

There is some related work and partial progress, however. A conjecture of Reed \cite{ReedConj} states that a graph $G$ with maximum degree $\Delta$ and clique number $\omega$ satisfies $\chi (G)\le \left\lceil (\Delta + 1 + \omega)/2 \right\rceil$. If this conjecture is true, then Alon and Kim's main result in \cite{AK97} would imply that a $k$-uniform $t$-simple hypergraph $\cH$ with maximum degree $D$ satisfies $\chi'(\cH) \le \left\lceil (k(D-1)+ 1 + (t-1+1/t)D)/2 \right\rceil$. Hurley, de Joannis de Verclos, and Kang \cite{HVK22} proved an epsilon-version of Reed's conjecture, in which $\chi(G) \le \left\lceil (1-\eps_1)(\Delta(G) + 1) + \eps_1(\omega) \right\rceil$ where $\eps_1=0.119$, and $\Delta(G)$ is sufficiently large. This together with Alon and Kim's result implies $\chi'(\cH) \le  \left\lceil (0.881)(k(D-1) + 1) + (0.119)(t-1+1/t)D \right\rceil$ for a $k$-uniform $t$-simple hypergraph $\cH$ with large maximum degree $D$.
In fact, one can do slightly better; a result of Kelly and Postle \cite[Theorem 4.1]{KP20} (see also \cite[Remark 3.1]{HVK22}) combined with Alon and Kim's bound on the clique number of the line graph of $\cH$ and the main result of Hurley, de Joannis de Verclos, and Kang \cite{HVK22} yields $\chi'(\cH) \leq (1 - \eps + o(1))kD$, where $\eps \in (0,1)$ satisfies $(1 - \eps) = 1 - (\alpha - \alpha^2/2 - 2\eps)/2 + (\alpha - \alpha^2/2 - 2\eps)^{3/2}/6$ for $\alpha = 1 - (t - 1 + 1/t)/k$.

The special case of $t=k$ in Conjecture \ref{Alon Kim} was conjectured earlier by F\"uredi, Kahn, and Seymour \cite{FKSfractional}, who verified that this case holds for the fractional chromatic index. A conjecture of Kahn \cite{KahnConj} states that the chromatic index of a $k$-uniform hypergraph is asymptotically equivalent to its fractional chromatic index, so the F\"uredi, Kahn, and Seymour result along with Kahn's conjecture, if true, confirms the Alon-Kim conjecture for $t=k$.
For more background on hypergraph coloring, see the survey by Kang, Kelly, K{\"u}hn, Methuku, and Osthus \cite{Colorsurvey}.

While Conjecture \ref{Alon Kim} remains open in general, this paper establishes new upper bounds on the chromatic index for hypergraphs of three specific classes: 3-uniform 2-simple, 3-uniform 3-simple, and $k$-uniform $k$-simple. 

\subsection{Chromatic index bounds for hypergraphs}

Using the general edge-coloring theorem in the next section, together with some analysis of local structure, we improve the bounds of the chromatic index for three specific cases below.
  
\begin{theorem}\label{thm: 3,2 corr}
    For each $\iota >0$, there is $D_0$ such that if $D\ge D_0$ and $\cH$ is a 3-uniform 2-simple hypergraph with maximum degree at most $D$, then $$\chi'(\cH) \le \left(1-\frac29 + \frac{2}{3^5} +\iota \right)3D.$$
    In particular, for sufficiently large $D$, every $3$-uniform $2$-simple hypergraph with maximum degree at most $D$ has chromatic index at most $2.3581D$.
\end{theorem}

\begin{theorem}\label{thm: 3,3 corr}
    For each $\iota >0$, there is $D_0$ such that if $D\ge D_0$ and $\cH$ is a 3-uniform 3-simple hypergraph with maximum degree at most $D$, then $$\chi'(\cH) \le \left(1-\frac19 + \frac{1}{3^5} +\iota \right)3D.$$
    In particular, for sufficiently large $D$, every $3$-uniform $3$-simple hypergraph with maximum degree at most $D$ has chromatic index at most $2.6791D$.
\end{theorem}

\begin{theorem}\label{thm: kk corr}
  For each $\iota >0$ and $k\in \mathbb{Z}^+$, there is $D_0$ such that if $D\ge D_0$ and $\cH$ is a $k$-uniform $k$-simple hypergraph with maximum degree at most $D$, then $$\chi'(\cH) \le \left(1-\frac{k-1}{2k^2} + \frac{(k-1)(k-2)}{6k^3\sqrt{k}} +\iota \right)kD.$$
\end{theorem}

For comparison, Conjecture~\ref{Alon Kim} predicts asymptotic bounds of $(3/2)D$ and $(7/3)D$ in the 3-uniform 2-simple and 3-uniform 3-simple cases, respectively. Our bounds are $2.3581D$ and $2.6791D$, improving the best bounds obtainable from previously known results, which are approximately $2.742D$ and $2.861D$. In the $k$-uniform $k$-simple case, Conjecture~\ref{Alon Kim} predicts a coefficient of $1-1/k+1/k^2$ in front of $kD$. Our bound has coefficient $1-1/(2k)+O(k^{-3/2})$, whereas the best previously known bound has coefficient $1-1/(4k)+O(k^{-3/2})$.

\subsection{A general edge-coloring theorem for hypergraphs}

The proofs of Theorems \ref{thm: 3,2 corr}, \ref{thm: 3,3 corr}, and \ref{thm: kk corr} are powered by a general edge-coloring theorem. This result provides an upper bound on the chromatic index of a hypergraph based on a local parameter $\theb$, defined for each edge via the structure of its neighborhood in the line graph. We begin by defining the key local parameters of a graph that will form the basis of our analysis.

    Let $G$ be a graph. For any vertex $v\in V(G)$, let $P_G(v)$ denote the \textit{number of independent pairs} in $N_G(v)$, which is defined as follows: $$ P_{G}(v) := \binom{deg_G(v)}{2}- |E(N_G(v))|.$$
    Let $T_G(v)$ denote the \textit{number of independent triples} in $N_G(v)$, which is defined as follows: $$T_{G}(v) := |\{ \{u_1,u_2,u_3\} \subseteq N_G(v) : u_i \nsim u_j \text{ for all } i,j\in [3] \}| .$$
    That is, $P_{G}(v)$ is the number of non-edges and $T_G(v)$ is the number of independent $3$-sets in the neighborhood of $v$.

For a graph $G$ and number $M \ge \Delta(G)$, we will use the parameter $b_{G,M}(v)$, which is defined as follows: $$b_{G,M}(v) :=\frac{\deg_G(v)}{M} -  \frac{P_{G}(v)}{M^2} + \frac{T_{G}(v)}{M^3}.$$
The terms in $b_{G,M}(v)$ arise from an inclusion-exclusion estimate for the number of colors made unavailable at $v$ during a random coloring procedure.
We will write $P(v)$, $T(v)$, and $b(v)$ when $G$ and $M$ are clear from context.

In our application, $G=L(\cH)$ is the line graph of a $k$-uniform hypergraph $\cH$, and we take $M=kD$, where $D\ge \Delta(\cH)$. Thus we will often consider $b_{L(\cH), kD}(e)$ where $e\in E(\cH)$, viewing the edge as a vertex in the line graph.

The following general theorem states that if the parameter $\theb(e)$ is bounded above by a constant $\theb$ for all edges in every sub-hypergraph of bounded maximum degree, then the chromatic index of the entire hypergraph is at most $(\theb + \eps)kD$.

\begin{theorem}\label{thm:gct}
      For all $\eps >0$ and $k\in \mathbb{Z}^{+}$, there exists $D_0$ such that the following holds for all $D\ge D_0$ and $\theb\in (0,1)$. Let $\cH$ be a $k$-uniform hypergraph with maximum degree at most $D$.
      If for all $D'\ge \eps D/3$ and $H\subseteq \cH$ with $\Delta(H)\le D',$ we have $$\theb_{L(H), kD'}(e) \le \theb$$
  for all $e\in E(H)$, then $$\chi'(\cH) \le (\theb+\eps)kD.$$
\end{theorem}

The point of Theorem~\ref{thm:gct} is that it separates the probabilistic coloring argument from the structural extremal analysis. 
In order to use this theorem to prove Theorems \ref{thm: 3,2 corr}, \ref{thm: 3,3 corr}, and \ref{thm: kk corr}, we will prove an upper bound on the parameter $\theb(e)$ in each of the three cases. For the former two cases, those of 3-uniform hypergraphs, our bound on the value of $b(e)$ is best possible.

\subsection{Outline of the paper}
Section 2 contains preliminary results and their proofs. Section 3 contains the proofs of Theorems \ref{thm: 3,2 corr}, \ref{thm: 3,3 corr}, and \ref{thm: kk corr}, which requires proving bounds on $\theb(e)$ in each of these three contexts. Section 4 contains the proof of our main theorem, Theorem \ref{thm:gct}.

Here we give a general outline of the proof of Theorem \ref{thm:gct}.
Our proof uses an iterative coloring procedure to color a few vertices from the line graph $L(\cH)$ at a time, where in each iteration we consider a random partial coloring and show with nonzero probability both the maximum degree and the number of colors made unavailable drop appropriately so that the nibble step of the coloring procedure can be repeated. Ultimately, when the maximum degree is small enough, we can finish coloring the vertices of $L(\cH)$ greedily.

To generate each color class of this partial coloring, we use an algorithm for sampling a random independent set in a regular graph introduced by Hurley, de Joannis de Verclos, and Kang \cite{HVK22} to improve bounds on the chromatic number of \textit{$\sigma$-sparse} graphs. We prove a generalization of their result about the probabilistic properties of this algorithm that can be applied to line graphs of uniform hypergraphs. Our nibble approach is inspired by theirs, but, as discussed in Section \ref{section: 4-gctproof}, we have modified the random partial coloring procedure which resulted in a simplified probabilistic analysis.

\section{Triangle Counting Lemmas and Other Tools}
This section has some useful structural tools, which will help us bound the parameter $\theb$ in the next section. We start with lemmas useful for counting the number of independent triples in the neighborhood of a vertex.

\begin{lemma} \label{lem: trianglecount1}
    A $k$-partite graph with $m$ edges has at most $\binom{k}{3} (m/\binom{k}{2})^{3/2}$ triangles. 
\end{lemma}

\begin{proof}
    Label the vertex partition of the graph as $A_1,A_2,...,A_k$, and let $T$ be the total number of triangles in the graph. Also define $T_i$ to be the number of triangles involving a vertex from $A_i$, and let $T_v$ be the number of triangles involving a particular vertex $v$. 
    Then $3T = \sum_{i\in [k]} T_i$ and $T_i = \sum_{v\in A_i} T_v$. 

    Let $v\in A_i$. 
    Since every triangle containing $v$ contains two neighbors of $v$ in distinct parts, we have $$T_v\le \sum_{\{j,j'\}\in\binom{[k]\setminus\{i\}}{2}} |N(v) \cap A_j| \cdot |N(v) \cap A_{j'}| \le \binom{k-1}{2}\left( \frac{\deg(v)}{k-1} \right)^2.$$
    Hence, 
    \begin{align*}
        \sqrt{T_v} &\le \sqrt{ \frac{k-2}{2(k-1)}}\deg(v).
    \end{align*}

    Let $m_i$ be the number of edges incident to a vertex in $A_i$. Thus $m_i = \sum_{v\in A_i} deg(v)$. Since every edge not incident to a vertex of $A_i$ is in at most one triangle containing $v\in A_i$, we have $T_v \le m-m_i$ for every $v\in A_i$. Thus $\sqrt{T_v} \le \sqrt{m-m_i}$.

    We can use our two bounds on $\sqrt{T_v}$ to bound the value of $T_i$ as follows:
    \begin{align*}
        T_i &= \sum_{v\in A_i} T_v = \sum_{v\in A_i} \sqrt{T_v} \sqrt{T_v} \le \sqrt{m-m_i} \sqrt{ \frac{k-2}{2(k-1)}} m_i.
    \end{align*}
    From here, we use $T = \frac13 \sum_{i\in [k]} T_i$ to bound $T$.
    Since the function $x\mapsto x\sqrt{m-x}$ is concave on $[0, m]$ and $\sum_{i=1}^k m_i =2m$, Jensen's Inequality implies that $$ \sum_{i=1}^k  m_i \sqrt{m-m_i} \le \sum_{i=1}^k \left( \frac{2m}{k}\right) \sqrt{m - \frac{2m}{k}} = 2m \sqrt{m-\frac{2m}{k}}.$$

    Thus, 
    \begin{align*}
        T &= \frac13 \sum_{i\in [k]} T_i 
        \le \frac13 \sqrt{ \frac{k-2}{2(k-1)}} \sum _{i\in [k]}\sqrt{m-m_i}  m_i 
        \le \frac{2m}{3}  \sqrt{ \frac{k-2}{2(k-1)}} \sqrt{m-\frac{2m}{k}}
        = \binom{k}{3} \sqrt{ \frac{m}{\binom{k}{2}}}^3 
    \end{align*}
    whence we get the desired result. 
\end{proof}

Note that Lemma \ref{lem: trianglecount1} implies that a $3$-partite graph with $m$ edges has at most $(m/3)^{3/2}$ triangles. 

\begin{lemma} \label{lem: trianglecount2}
    A tripartite graph $G$  has at most $\frac{1}{12} \sum_{v\in V(G)} \deg(v)^2$ triangles.
\end{lemma}

\begin{proof}
    Label the vertex partition of the graph as $A_1,A_2,A_3$, and let $T_v$ be the number of triangles involving a particular vertex $v\in V(G)$. Then $3T = \sum_{v\in V(G)} T_v$. 

    For every vertex $v\in A_i$, since every triangle containing $v$ contains two neighbors of $v$ in distinct parts, we have $$ T_v \le |N(v) \cap A_j|\cdot |N(v) \cap A_{j'}| \le \left( \frac{\deg(v)}{2}\right)^2,$$
    where ${\{j,j'\}=[3]\setminus\{i\}}$.
    Thus $T = \frac13 \sum_{v\in V(G)} T_v \le \frac1{12} \sum_{v\in V(G)} \deg(v)^2$ as desired.
\end{proof}

\begin{fact}\label{lem: sums}
    For any real numbers $r_1,r_2,...,r_n \in [a,b]$, where $a\le b$ are real numbers, we have $\sum_{i=1}^n r_i^2 \le (a+b) (\sum_{i=1}^n r_i) - nab$.
\end{fact}

\begin{proof}
    Using the following inequality $$0\le (r_i-a)(b-r_i) = (a+b)r_i - r_i^2 - ab,$$
    we get $\sum_{i=1}^n r_i^2 \le (a+b) (\sum_{i=1}^n r_i) - nab$.
\end{proof}

The next two lemmas will aid in determining what structure in a hypergraph will optimize the parameter $\theb$. Given $e\in E(\cH)$, we will be interested in the structure of $N_{L(\cH)}(e)$ which makes $\theb_{L(\cH), kD}(e)$ large, or the related goal to make $P_{L(\cH)}(e)$ small. In context, we will usually have $a_{i,j}$ in the lemmas denote the codegree between vertices $v_i\in e$ and $x_j\notin e$. But for this section, the next two lemmas are stated generally, as applications of other common algebraic inequalities.

\begin{fact}\label{lem: ops1}
Let $k\in \mathbb{N}$. If $a_1,...,a_k\in \mathbb{R}$, then $$\sum_{1\le i < j \le k} a_ia_j=\frac{k-1}{2k} \left(\sum_{i=1}^k a_i \right)^2-\frac{1}{2k} \sum_{1\le i < j \le k} (a_i-a_j)^2 \le \frac{k-1}{2k} \left(\sum_{i=1}^k a_i \right)^2.$$
\end{fact}

\begin{proof}
    The equality is trivial, and the inequality follows since $(a_i-a_j)^2 \ge 0$ for all $i$ and $j$.
\end{proof}

\begin{lemma} \label{lem: tableops redo}
  Let $k, n \in \mathbb{N}$ and $C, T \in \mathbb{R}^{>0}$, and let $a_{i,j} \in \mathbb{R}^{\ge 0}$ for every $i \in [k]$ and $j \in [n]$. If $\sum_{i=1}^{k} a_{i,j} \le C$ for every  $j \in [n]$ and $\sum_{i=1}^{k} \sum_{j=1}^{n} a_{i,j} \le T$, then
\[ \sum_{j=1}^{n} \sum_{1 \le i < i' \le k} a_{i,j} \cdot a_{i',j} \le \frac{k-1}{2k}(qC^2 + r^2) \le \frac{k-1}{2k} TC, \text{ where }  q= \lfloor T/C \rfloor \text{ and } r= T - qC.\] 
Moreover, if there exist distinct $j_1, j_2 \in [n]$ and $i^* \in [k]$ such that 
\[\left|\sum_{i \in [k] \setminus \{i^*\}} (a_{i,j_1} + a_{i,j_2})- 2\left(1 - \frac{1}{k}\right)C\right|\ge \delta \text{ for some } \delta > 0,\] then
\begin{equation*}
    \sum_{j=1}^{n} \sum_{1 \le i < i' \le k} a_{i,j} a_{i',j} \le \frac{k-1}{2k}nC^2 - \frac{k\delta^2}{4(k-1)}.
\end{equation*}
\end{lemma}

\begin{proof}

    Without loss of generality, suppose the indices $j\in [n]$ are ordered such that $\sum_{i=1}^k a_{i,1} \ge \sum_{i=1}^k a_{i,2} \ge \cdots \ge \sum_{i=1}^k a_{i,n}$. We may also assume that $C\le T$ and $\sum_{i=1}^k \sum_{j=1}^n a_{i,j} =T$.
    
    Let us use Karamata's inequality applied to the real-valued convex function $f(x)=x^2$.  Define a sequence $Y:=(y_j)_{j=1}^n$ where $y_j = C$ for $j\in [q]$, and $y_{q+1}=r$, and all remaining terms (if any) are zero. In order to apply the inequality, we first show that the sequence $Y$ majorizes the sequence $(\sum_{i=1}^k a_{i,j})_{j=1}^n$.  
    
    Indeed, because $\sum_{i=1}^k a_{i,j} \le C$ for every $j\in [n]$, and because $\sum_{i=1}^k \sum_{j=1}^n a_{i,j} \le qC +r$ by choice of $q$ and $r$, we have $\sum_{i=1}^k \sum_{j=1}^{\ell} a_{i,j} \le \sum_{j=1}^{\ell} y_j$ for each $\ell\in [n]$. Thus, Karamata's inequality implies that $\sum_{j=1}^n f(\sum_{i=1}^k a_{i,j}) \le \sum_{j=1}^n f(y_j)$, that is, $$\sum_{j=1}^n \left(\sum_{i=1}^k a_{i,j}\right)^2 \le qC^2 + r^2. $$
    This, together with Fact \ref{lem: ops1}, implies
    \begin{equation*}
        \sum_{j=1}^n \sum_{1\le i < i' \le k} a_{i,j}a_{i',j} 
        \le \frac{k-1}{2k}\sum_{j=1}^n \left(\sum_{i=1}^k a_{i,j}\right)^2 
        \le \frac{k-1}{2k}(qC^2+r^2), 
    \end{equation*}
    as desired. In addition, since $r \leq C$, we have
    \begin{equation*}
      \frac{k-1}{2k}(qC^2+r^2) \le \frac{k - 1}{2k}(qC^2 + rC) = \frac{k-1}{2k} TC.
    \end{equation*}

    For the ``moreover'' part, without loss of generality, we assume that $j_1 = n - 1$, $j_2 = n$, and $i^* = k$.
    For notational convenience, we let $\alpha_i = a_{i,n-1}$ and $\beta_i = a_{i,n}$ for each $i \in [k-1]$, we let $A = a_{k,n-1}$ and $B = a_{k,n}$, and we let $S_\alpha = \sum_{i=1}^{k-1}\alpha_i$ and $S_\beta = \sum_{i=1}^{k-1}\beta_i$.
    Note that
    \begin{equation*}
      \sum_{j=1}^n\sum_{1 \leq i < i' \leq k}a_{i,j}a_{i',j} = \sum_{j=1}^{n-2}\sum_{1 \leq i < i' \leq k}a_{i,j}a_{i',j} + \sum_{1 \leq i < i' \leq k - 1}(\alpha_i \alpha_{i'} + \beta_i \beta_{i'}) + \sum_{i=1}^{k-1}(A\alpha_i + B\beta_i).
    \end{equation*}
    By Fact \ref{lem: ops1},
    we have $\sum_{j=1}^{n-2}\sum_{1 \leq i < i' \leq k}a_{i,j}a_{i',j} \leq (n - 2)\frac{k-1}{2k}C^2$, so to finish the proof, it suffices to show that
    \begin{equation*}
      \sum_{1 \leq i < i' \leq k - 1}(\alpha_i \alpha_{i'} + \beta_i \beta_{i'}) + \sum_{i=1}^{k-1}(A\alpha_i + B\beta_i) \leq 2C^2\frac{k - 1}{2k} - \frac{k\delta^2}{4(k-1)}.
    \end{equation*}
    To that end, we let
    \begin{equation*}
      E := 2C^2\frac{k-1}{2k}-\frac{k\delta^2}{4(k-1)}-\Big(\sum_{1\le i<i'\le k-1}(\alpha_i\alpha_{i'}+\beta_i\beta_{i'})+\sum_{i=1}^{k-1}(A\alpha_i+B\beta_i)\Big),
    \end{equation*}
    and we show that $E \geq 0$.
    
    By assumption, $\delta \leq |(S_\alpha + S_\beta) - \frac{2(k-1)}{k}C|$, so by squaring both sides of this inequality, we have
    \begin{equation}
      \label{eq:2.1}
      \frac{k}{4(k - 1)}\delta^2 \leq \frac{k}{4(k-1)}(S_\alpha + S_\beta)^2 - C(S_\alpha + S_\beta) + \frac{k - 1}{k}C^2.
    \end{equation}
    By Fact~\ref{lem: ops1}  again, since $A \leq C - S_\alpha$,
    \begin{align}
      \sum_{1\le i<i'\le k-1}\alpha_i\alpha_{i'}+\sum_{i=1}^{k-1}A\alpha_i
      &=\frac{k-2}{2(k-1)}S_\alpha^2 - \frac{1}{2(k-1)}\sum_{i<i'}(\alpha_i - \alpha_{i'})^2 + A S_\alpha \nonumber\\
      &\le \frac{k-2}{2(k-1)}S_\alpha^2- \frac{1}{2(k-1)}\sum_{i<i'}(\alpha_i - \alpha_{i'})^2+(C - S_\alpha) S_\alpha \nonumber\\
      &\le CS_\alpha-\frac{k}{2(k-1)}S_\alpha^2 - \frac{1}{2(k-1)}\sum_{i<i'}(\alpha_i - \alpha_{i'})^2 \nonumber \\
      &\le CS_\alpha-\frac{k}{2(k-1)}S_\alpha^2. \label{eq:2.3}
    \end{align}
    Similarly, since $B \leq C - S_\beta$,
    \begin{align}\label{eq:2.4}
      \sum_{1\le i<i'\le k-1}\beta_i\beta_{i'}+\sum_{i=1}^{k-1}B\beta_i 
      & \le CS_\beta-\frac{k}{2(k-1)}S_\beta^2 - \frac{1}{2(k-1)}\sum_{i<i'}(\beta_i - \beta_{i'})^2 \nonumber \\
      & \le CS_\beta-\frac{k}{2(k-1)}S_\beta^2. 
    \end{align}
    Combining \eqref{eq:2.1}, \eqref{eq:2.3}, and \eqref{eq:2.4}, we have
    \begin{align*}E&\ge \frac{k-1}{k}C^2 - \frac{k\delta^2}{4(k-1)} - C(S_\alpha + S_\beta) + \frac{k}{2(k-1)}(S_\alpha^2 + S_\beta^2)\\
	&= \frac{k-1}{k}C^2 - \frac{k\delta^2}{4(k-1)} - C(S_\alpha + S_\beta) + \Big(\frac{k}{4(k-1)}(S_\alpha + S_\beta)^2 + \frac{k}{4(k-1)}(S_\alpha - S_\beta)^2\Big)\\
	&= \Big(\frac{k-1}{k}C^2 - C(S_\alpha + S_\beta) + \frac{k}{4(k-1)}(S_\alpha + S_\beta)^2\Big) - \frac{k\delta^2}{4(k-1)} + \frac{k}{4(k-1)}(S_\alpha - S_\beta)^2\\
&\ge\frac{k}{4(k-1)}(S_\alpha - S_\beta)^2\ge 0,
\end{align*}
    as desired.
\end{proof}

The \textit{permanent} of an $n\times n$ matrix $M=(m_{i,j})$ is defined as $$\operatorname{perm}(M):= \sum_{\sigma \in S_n} \prod_{i=1}^n m_{i, \sigma(i)}$$ where $S_n$ is the symmetric group, that is, the group of all permutations of $[n]$.

The following lemma involving the permanent of $3\times n$ matrix will be useful for bounding the number of independent triples for an edge in a $3$-uniform $2$-simple hypergraph with few independent pairs.

\begin{lemma}
\label{lem:3xn-matrix-bound}
Let $n \geq 3$, let $A = (a_{i,j})_{i\in[3],j\in[n]}$ be a $3 \times n$ matrix with entries $a_{i,j} \in [0,1]$, and let $c_j := \sum_{i=1}^3 a_{i,j} \leq 1$ for each $j \in [n]$.
If  $\sum_{j=1}^n c_j \leq 3$ and $1 \geq c_1 \geq c_2 \geq \cdots \geq c_n$, then
\begin{equation*}
  \frac{1}{9}\sum_{j=1}^{n} \sum_{1 \leq i < i' \leq 3} a_{i,j}a_{i',j}
  + \frac{1}{27} \operatorname{perm}(A_{[3]})
  - \frac{1}{27} \sum_{i=1}^{3}\sum_{j=1}^{3} a_{i,j} \leq \frac{2}{3^5},
\end{equation*}
where $A_{[3]}$ denotes the leading $3 \times 3$ submatrix of $A$.
\end{lemma}

First let's prove the case when it's a $3 \times 3$ matrix, for a fixed value of the column sums. Note that equality is achieved when all entries equal $1/3$.

\begin{lemma}\label{lem:square}
  Let $A = (a_{i,j})_{i,j\in[3]}$ be a $3 \times 3$ matrix with entries $a_{i,j} \in [0,1]$, and let $c_j \coloneqq \sum_{i=1}^3 a_{i,j}$ for each $j \in [3]$. If $c_1,c_2,c_3 \leq 1$, then
  \begin{equation*}
      3\sum_{j=1}^{3} \sum_{1 \leq i < i' \leq 3} a_{i,j}a_{i',j}
      + \operatorname{perm}(A) \leq c_1^2 + c_2^2 + c_3^2 + \frac{2}{9}c_1c_2c_3.
  \end{equation*}
\end{lemma}
\begin{proof}
  Let $A$ be a $3\times 3$ matrix satisfying the hypotheses with
  \begin{equation*}
    3\sum_{j=1}^{3} \sum_{1 \leq i < i' \leq 3} a_{i,j}a_{i',j}
      + \operatorname{perm}(A) 
    \end{equation*}
    maximum. (Since the function is continuous over a compact domain, such a maximizer exists.)

    It suffices to show that each entry in column $j$ equals $c_j / 3$, in which case we can compute
    \begin{equation*}
      3\sum_{j=1}^{3} \sum_{1 \leq i < i' \leq 3} a_{i,j}a_{i',j}
      + \operatorname{perm}(A) = c_1^2 + c_2^2 + c_3^2 + \frac{2}{9}c_1c_2c_3,
    \end{equation*}
    as desired.
    
    Suppose not, and consider two entries in the same column with the greatest difference. We assume without loss of generality it's $a_{1,3}$ and $a_{3,3}$, so $a_{3,3} - a_{1,3} \geq a_{i,j} - a_{i',j}$ for all $i,i',j \in [3]$.
    Let $\eps > 0$ be a sufficiently small number where in particular $\eps < (a_{3,3} - a_{1,3})/3$, and let $A' = (a'_{i,j})_{i,j\in[3]}$ be the matrix obtained from $A$ by adding $\eps$ to $a_{1,3}$ and subtracting $\eps$ from $a_{3,3}$.

    By the choice of $A$, we have
    \begin{equation*}
      3\sum_{j=1}^{3} \sum_{1 \leq i < i' \leq 3} a_{i,j}a_{i',j}
      + \operatorname{perm}(A) \geq     3\sum_{j=1}^{3} \sum_{1 \leq i < i' \leq 3} a'_{i,j}a'_{i',j}
      + \operatorname{perm}(A'),
    \end{equation*}
    but we compute
    \begin{multline*}
      \left(3\sum_{j=1}^{3} \sum_{1 \leq i < i' \leq 3} a'_{i,j}a'_{i',j}
        + \operatorname{perm}(A')\right)  -  \left(3\sum_{j=1}^{3} \sum_{1 \leq i < i' \leq 3} a_{i,j}a_{i',j}
        + \operatorname{perm}(A)\right)\\
      = 3\eps\left( a_{3,3} - a_{1,3} \right) - 3\eps^2 + \eps\left( a_{2,1}a_{3,2} + a_{2,2}a_{3,1} - a_{1,1} a_{2,2} - a_{1,2}a_{2,1}  \right).
    \end{multline*}
    By the choice of $a_{1,3}$ and $a_{3,3}$, we have $a_{3,2} - a_{1,2}, a_{3,1} - a_{1,1} \geq -(a_{3,3} - a_{1,3})$, so
    \begin{align*}
      a_{2,1}a_{3,2} + a_{2,2}a_{3,1} - a_{1,1} a_{2,2} - a_{1,2}a_{2,1}
      &= a_{2,1}(a_{3,2} - a_{1,2}) + a_{2,2}(a_{3,1} - a_{1,1})\\
      &\geq -(a_{2,1} + a_{2,2})(a_{3,3} - a_{1,3}) \geq -2(a_{3,3} - a_{1,3}).   
    \end{align*}
    Combining the inequalities above, we have
    \begin{equation*}
      \eps(a_{3,3} - a_{1,3}) - 3\eps^2 \leq 0,
    \end{equation*}
    contradicting our choice of $A$ and $\eps$.
\end{proof}

We need one more helpful lemma.
\begin{lemma}\label{lem:sum-of-squares-minus-S}
  If $1 \geq c_1 \geq c_2 \geq \cdots \geq c_n \geq 0$ and $\sum_{j=1}^n c_j \leq 3$, then
  \begin{equation*}
    \sum_{j=1}^n c_j^2 + \frac{2}{9}c_1c_2c_3 - c_1 - c_2 - c_3 \leq \frac{2}{9}.
  \end{equation*}
\end{lemma}
\begin{proof}
  First, we claim that
  \begin{equation}\label{eq:sum-of-squares-at-least-four}
    \sum_{j=4}^n c_j^2 \leq c_3(3 - c_1 - c_2 - c_3).
  \end{equation}
  Indeed, since $c_3 \geq c_j \geq 0$ for all $j \geq 4$, we have
  \begin{equation*}
    \sum_{j=4}^n c_j^2 \leq \sum_{j=4}^nc_3 c_j = c_3\sum_{j=4}^n c_j,
  \end{equation*}
  and since $\sum_{j=1}^n c_j \leq 3$, we have 
  \begin{equation*}
    \sum_{j=4}^n c_j \leq 3 - c_1 - c_2 - c_3.
  \end{equation*}
  Combining the inequalities above yields \eqref{eq:sum-of-squares-at-least-four}.
  
  Therefore,
  \begin{align*}
    \sum_{j=1}^n c_j^2 + \frac{2}{9}c_1c_2c_3 - c_1 - c_2 - c_3
    &\leq c_1^2 + c_2^2 + c_3^2 + 3c_3 - c_1c_3 - c_2c_3 - c_3^2 + \frac{2}{9}c_1c_2c_3 - c_1 - c_2 - c_3\\
    &= c_1(c_1 - 1) + c_2(c_2 - 1) + c_3(2 - c_1 - c_2 + \frac{2}{9}c_1c_2).
  \end{align*}

  Now let $f(x,y,z) \coloneqq x(x - 1) + y(y - 1) + z(2 - x - y + \frac{2}{9}xy)$. By the previous inequality, it suffices to prove that $f(x,y,z) \leq 2/9$ for all $1 \geq x \geq y \geq z \geq 0$.
  
  Since $f$ is linear in $z$ and $2 - x - y + \frac{2}{9}xy \geq 0$ for $0 \leq x,y\leq 1$, we have 
  \begin{equation*}
    f(x,y,z) \leq f(x, y, y) = x(x - 1) + y(1 - x + \frac{2}{9}xy).
  \end{equation*}
  Let $g(x,y) \coloneqq x(x - 1) + y(1 - x + \frac{2}{9}xy)$. Now we have
  \begin{equation*}
    \frac{\partial g}{\partial y} = 1 - x + \frac49xy \geq 0,
  \end{equation*}
  so $g(x,y) \leq g(x,x)$. Hence,
  \begin{equation*}
    f(x,y,z) \leq g(x,x) = x(x - 1) + x(1 - x + \frac{2}{9}x^2) = \frac{2}{9}x^2 \leq \frac{2}{9},
  \end{equation*}
  as desired.
\end{proof}

Now we can prove Lemma~\ref{lem:3xn-matrix-bound}.

\begin{proof}[Proof of Lemma~\ref{lem:3xn-matrix-bound}]
  For notational convenience, for a $3 \times n$ matrix $A$, define
\[
  Q(A) \coloneqq \sum_{j=1}^{n} \sum_{1 \leq i < i' \leq 3} a_{i,j}a_{i',j},
  \qquad
  P(A) \coloneqq \operatorname{perm}(A_{[3]}),
  \qquad
  S(A) \coloneqq \sum_{i=1}^{3}\sum_{j=1}^{3} a_{i,j},
\]
and let
\begin{equation*}
  F(A) \coloneqq 3Q(A) + P(A) - S(A).
\end{equation*}
Let $A = (a_{i,j})$ be a $3 \times n$ matrix satisfying the hypotheses with $F(A)$ maximum (such a maximum exists because $F$ is a continuous function over a compact domain). 
It suffices to show that $F(A) \leq 2/9$.
Let $Q_1 \coloneqq \sum_{j=1}^{3} \sum_{1 \leq i < i' \leq 3} a_{i,j}a_{i',j}$ and $Q_2 \coloneqq \sum_{j=4}^{n} \sum_{1 \leq i < i' \leq 3} a_{i,j}a_{i',j}$.

By Fact \ref{lem: ops1},
\begin{equation*}
  Q_2 \leq \frac{1}{3}\sum_{j=4}^n c_j^2.
\end{equation*}
We also have by Lemma~\ref{lem:square},
\begin{equation*}
  3Q_1 + P(A) \leq c_1^2 + c_2^2 + c_3^2 + \frac{2}{9}c_1c_2c_3.
\end{equation*}
Therefore,
\begin{equation*}
  F(A) \leq \sum_{j=1}^nc_j^2 + \frac{2}{9}c_1c_2c_3 - S(A),
\end{equation*}
and the desired result follows from Lemma~\ref{lem:sum-of-squares-minus-S}.
\end{proof}

\section{Coloring $3$-uniform and $k$-uniform hypergraphs}

The goal of this section is to prove Theorems \ref{thm: 3,2 corr}, \ref{thm: 3,3 corr}, and \ref{thm: kk corr}. The strategy is to apply the general framework of Theorem \ref{thm:gct}. This requires proving, for each class of hypergraphs considered, a uniform upper bound on the parameter $\theb_{L(\cH), kD}(e)$ for all edges $e$ in any sub-hypergraph. We establish these bounds through a series of structural lemmas.

It will be useful to have the following notation.

\begin{definition}
    Let $H(k,t,D)$ be the set of $k$-uniform, $t$-simple hypergraphs $\cH$ with $\Delta(\cH) \le D$. 
\end{definition}

\subsection{Upper bounds on the $\theb$ parameter}

The following lemmas provide the crucial upper bounds for the $\theb$ parameter in our three main cases. Their proofs, which rely on the subsequent structural analysis, are given in subsections \ref{subsection: 3,2 a param proof}, \ref{subsection: 3,3 param proof}, and \ref{subsection: kk param proof}.

The first two lemmas concern the 3-uniform cases, and give $\theb$ parameters which are best possible up to the $o(1)$ term.

\begin{lemma}\label{lem: 3,2 a parameter}
  Let $\cH\in H(3,2,D)$, and let $e\in E(\cH)$ be an edge. Then, as $D \to \infty$,
  $$\theb_{L(\cH), 3D}(e) \le 1 - \frac{2}{9} + \frac{2}{3^5} + o(1).$$
\end{lemma}

\begin{lemma}\label{lem: 3,3 a parameter}
  Let $\cH\in H(3,3,D)$, and let $e\in E(\cH)$ be an edge. Then, as $D \to \infty$,
  $$\theb_{L(\cH), 3D}(e) \le 1-\frac{1}{9} + \frac{1}{3^5} +o(1).$$
\end{lemma}

\begin{lemma}\label{lem: kk a parameter}
  Let $\cH\in H(k,k,D)$, and let $e\in E(\cH)$ be an edge. Then, as $D \to \infty$,
  $$\theb_{L(\cH), kD}(e) \le 1-\frac{k-1}{2k^2} + \frac{(k-1)(k-2)}{6k^3\sqrt{k}} + o(1).$$
\end{lemma}

Let's discuss some intuition behind the constants in the expressions above. Given the definition of the parameter $\theb_{L(\cH), 3D}(e)$, we might guess that $\theb (e)$ is largest when $\deg_{L(\cH)}(e)$ is as large as possible, and subject to this, when the number of independent pairs $P_{L(\cH)}(e)$ is as small as possible. 

For a 3-uniform 2-simple hypergraph, this means we look for an edge $e$ where $\deg_{L(\cH)}(e) = 3D-3$, the largest possible value, and there are many edges in $N_{L(\cH)}(e)$. The edge $e$ in Figure \ref{fig: 3,2 optimum edge} is therefore a good candidate.

Indeed, as we show in the next lemma, the configuration in Figure \ref{fig: 3,2 optimum edge} realizes the bound, up to a vanishing error, in Lemma~\ref{lem: 3,2 a parameter}.

\begin{lemma}\label{lem: extreme 3,2}
    For every $D \in \mathbb Z^+$, there exists a 3-uniform 2-simple hypergraph $\cH$ of maximum degree at most $D$ and an edge $e \in \cH$ such that
  \begin{equation*}
    b_{L(\cH),3D}(e) \geq 1 - \frac{2}{9} + \frac{2}{3^5} -o(1),
  \end{equation*} 
  where $o(1)$ goes to 0 as $D$ goes to infinity.
\end{lemma}

\begin{proof}
    Let $D'\in \{D-2, D-1, D\}$ such that $D'-1$ is divisible by 3.
    
    Construct a 3-uniform 2-simple hypergraph with maximum degree $D'$ as follows. Let the vertex set be $V(\cH) = \{v_1,v_2,v_3\}\cup \{x_1,x_2,x_3\}\cup \{y_{i,j,1}, y_{i,j,2},...,y_{i,j,(D'-1)/3}: (i,j)\in [3]\times [3]\}$. Let $e\in \cH$ be an edge such that $e=\{v_1,v_2,v_3\}$.

    For every $(i,j)\in [3]\times [3]$ and $k\in [(D'-1)/3]$, we also have the edge $\{v_i,x_j, y_{i,j,k}\}$. Thus, there are $(D'-1)/3$ edges containing the pair $\{v_i,x_j\}$. See Figure \ref{fig: 3,2 optimum edge}.

    We calculate $\deg_{L(\cH)}(e) = 3D'-3=3D+o(D)$, $P_{L(\cH)}(e) = 18 ((D'-1)/3)^2 = 2D^2+o(D^2)$, and $T_{L(\cH)}(e) = 6 ((D'-1)/3)^3 = \frac29 D^3 + o(D^3)$, which altogether gives $b_{L(\cH),3D}(e)=\frac{3D+o(D)}{3D}-\frac{2D^2 + o(D^2)}{(3D)^2}+\frac{\frac29 D^3 + o(D^3)}{(3D)^3} = 1-\frac29 + \frac{2}{3^5}+o(1)$.
\end{proof}

\begin{figure}[t]
    \centering
    \begin{subfigure}[t]{0.48\textwidth}
	    \centering
	    \resizebox{0.85\linewidth}{!}{%
	    \begin{tikzpicture}

		\tikzstyle{every node}=[font=\LARGE]
	\path[use as bounding box] (4.25,3.5) rectangle (13.25,13.8);
	\node at (8.75,12) [circ] {};
\node at (8.75,4.5) [circ] {};
\node at (5,12) [circ] {};
\node at (5,4.5) [circ] {};
\node at (12.5,12) [circ] {};
\node at (12.5,4.5) [circ] {};
\draw [line width=0.6pt, short] (5,12) .. controls (4.5,8.25) and (4.5,8.25) .. (5,4.5);
\draw [line width=0.6pt, short] (5,12) .. controls (6.5,7.75) and (6.5,8) .. (8.75,4.5);
\draw [line width=0.6pt, short] (5,12) .. controls (8.5,7.75) and (8.5,7.75) .. (12.5,4.5);
\draw [line width=0.6pt, short] (8.75,12) .. controls (6.75,8.5) and (6.5,8.5) .. (5,4.5);
\draw [line width=0.6pt, short] (8.75,12) .. controls (8.25,8.25) and (8,8.25) .. (8.75,4.5);
\draw [line width=0.6pt, short] (12.5,12) .. controls (12,8.25) and (11.75,8.25) .. (12.5,4.5);
\draw [line width=0.6pt, short] (8.75,12) .. controls (10.25,8.25) and (10.25,8.25) .. (12.5,4.5);
\draw [line width=0.6pt, short] (12.5,12) .. controls (10.25,8.25) and (10,8.25) .. (8.75,4.5);
\draw [line width=0.6pt, short] (12.5,12) .. controls (8,8.25) and (8,8.25) .. (5,4.5);
\draw [line width=0.6pt, short] (5,12) -- (5,12);
\draw [line width=0.6pt, short] (5,12) .. controls (5.5,8.25) and (5.5,8.25) .. (5,4.5);
\draw [line width=0.6pt, short] (5,12) .. controls (7.25,8.25) and (7.25,8.25) .. (8.75,4.5);
\draw [line width=0.6pt, short] (5,12) .. controls (9,8.25) and (9.25,8.25) .. (12.5,4.5);
\draw [line width=0.6pt, short] (8.75,12) .. controls (7.25,8.25) and (7.25,8.25) .. (5,4.5);
\draw [line width=0.6pt, short] (8.75,12) .. controls (9.25,8.25) and (9.25,8) .. (8.75,4.5);
\draw [line width=0.6pt, short] (8.75,12) .. controls (11.25,8.25) and (11.25,8.25) .. (12.5,4.5);
\draw [line width=0.6pt, short] (12.5,12) .. controls (9.25,8.25) and (9.25,8.25) .. (5,4.5);
\draw [line width=0.6pt, short] (12.5,12) .. controls (11.25,8.25) and (11.25,8.25) .. (8.75,4.5);
\draw [line width=0.6pt, short] (12.5,12) .. controls (13,8.25) and (13,8.25) .. (12.5,4.5);
\draw [line width=0.6pt, short] (5,12) -- (12.5,12);
\node [font=\LARGE] at (5,12.5) {$v_1$};
\node [font=\LARGE] at (8.75,12.5) {$v_2$};
\node [font=\LARGE] at (12.5,12.5) {$v_3$};
	\node [font=\LARGE] at (5,3.75) {$x_1$};
	\node [font=\LARGE] at (8.75,3.75) {$x_2$};
	\node [font=\LARGE] at (12.5,3.75) {$x_3$};
\draw [line width=0.5pt, short] (5,12) -- (5,4.5);
\draw [line width=0.5pt, short] (5,12) -- (8.75,4.5);
\draw [line width=0.5pt, short] (5,12) -- (12.5,4.5);
\draw [line width=0.5pt, short] (8.75,12) -- (5,4.5);
\draw [line width=0.5pt, short] (8.75,12) -- (8.75,4.5);
\draw [line width=0.5pt, short] (8.75,12) -- (12.5,4.5);
\draw [line width=0.5pt, short] (12.5,12) -- (5,4.5);
\draw [line width=0.5pt, short] (12.5,12) -- (8.75,4.5);
\draw [line width=0.5pt, short] (12.5,12) -- (12.5,4.5);
	\node [font=\LARGE] at (8.75,13.5) {$e$};

    \end{tikzpicture}%
    }
    \caption{A 3-uniform 2-simple hypergraph}
    \label{fig: 3,2 optimum edge}
    \end{subfigure}\hfill
    \begin{subfigure}[t]{0.48\textwidth}
	    \centering
	    \resizebox{0.85\linewidth}{!}{%
		\begin{tikzpicture}
		\tikzstyle{every node}=[font=\LARGE]
		\path[use as bounding box] (4.25,3.5) rectangle (13.25,13.8);
		\node at (8.75,12) [circ] {};
	\node at (5,12) [circ] {};
	\node at (12.5,12) [circ] {};
	\node at (8.75,8.25) [circ] {};
	\node at (5,8.25) [circ] {};
	\node at (12.5,8.25) [circ] {};
	\node at (5,4.5) [circ] {};
	\node at (8.75,4.5) [circ] {};
	\node at (12.5,4.5) [circ] {};
	\draw [line width=0.5pt, short] (5,12) -- (5,4.5);
	\draw [line width=0.5pt, short] (8.75,12) -- (8.75,4.5);
	\draw [line width=0.5pt, short] (12.5,12) -- (12.5,4.5);
	\draw [line width=0.5pt, dashed] (5,12) -- (12.5,4.5);
	\draw [line width=0.5pt, dashed] (8.75,12) .. controls (11.25,10.25) and (12.25,9.5) .. (12.5,8.25);
	\draw [line width=0.5pt, dashed] (12.5,8.25) .. controls (12.25,6) and (9.25,5.75) .. (5,4.5);
	\draw [line width=0.5pt, dashed] (12.5,12) .. controls (8,10.25) and (6,10.5) .. (5,8.25);
	\draw [line width=0.5pt, dashed] (5,8.25) .. controls (5.25,6.5) and (6.75,5.75) .. (8.75,4.5);
	\draw [ line width=0.5pt, short, dotted] (8.75,12) .. controls (6.5,10.75) and (5.5,10.25) .. (5,8.25);
	\draw [ line width=0.5pt, short, dotted] (5,8.25) .. controls (5.5,6) and (8.25,5.5) .. (12.5,4.5);
	\draw [line width=0.5pt, short, dotted] (12.5,12) -- (5,4.5);
	\draw [ line width=0.5pt, short, dotted] (5,12) .. controls (8,11) and (12.4,10) .. (12.5,8.25);
	\draw [ line width=0.5pt, short, dotted] (12.5,8.25) .. controls (12.5,5.25) and (11.25,5) .. (8.75,4.5);
	\draw [line width=0.5pt, short] (5,12) -- (12.5,12);
	\node [font=\LARGE] at (5,12.5) {$v_1$};
	\node [font=\LARGE] at (8.75,12.5) {$v_2$};
	\node [font=\LARGE] at (12.5,12.5) {$v_3$};
	\node [font=\LARGE] at (5,3.75) {$x_1$};
	\node [font=\LARGE] at (8.75,3.75) {$x_2$};
	\node [font=\LARGE] at (12.5,3.75) {$x_3$};
	\node [font=\LARGE] at (4.25,8.25) {$x_4$};
	\node [font=\LARGE] at (7.75,8.25) {$x_5$};
	\node [font=\LARGE] at (13.25,8.25) {$x_6$};
	\node [font=\LARGE] at (8.75,13.5) {$e$};
	\end{tikzpicture}%
    }
    \caption{A 3-uniform 3-simple hypergraph}
    \label{fig: 3,3 optimum edge}
    \end{subfigure}
    \caption{Extremal examples. In (a), edge $e=\{v_1,v_2,v_3\}$ intersects $3D-3$ other edges, each of which contains exactly one of $\{x_1,x_2,x_3\}$. For each $i\in [3]$ and $j\in [3]$, there are $(D-1)/3$ edges containing $\{v_i,x_j\}$; the remaining vertices of all these edges are distinct and are not shown. In (b), edge $e=\{v_1,v_2,v_3\}$ intersects $3D-3$ other edges, each of which contains exactly one of $\{x_1,x_2,x_3\}$ and exactly one of $\{x_4,x_5,x_6\}$. Aside from $e$, there are 9 types of edges, each one occurring $(D-1)/3$ times. Such edges of the same type (solid, dashed, or dotted) form a size 3 matching.}
    \label{fig: optimum edge examples}
\end{figure}
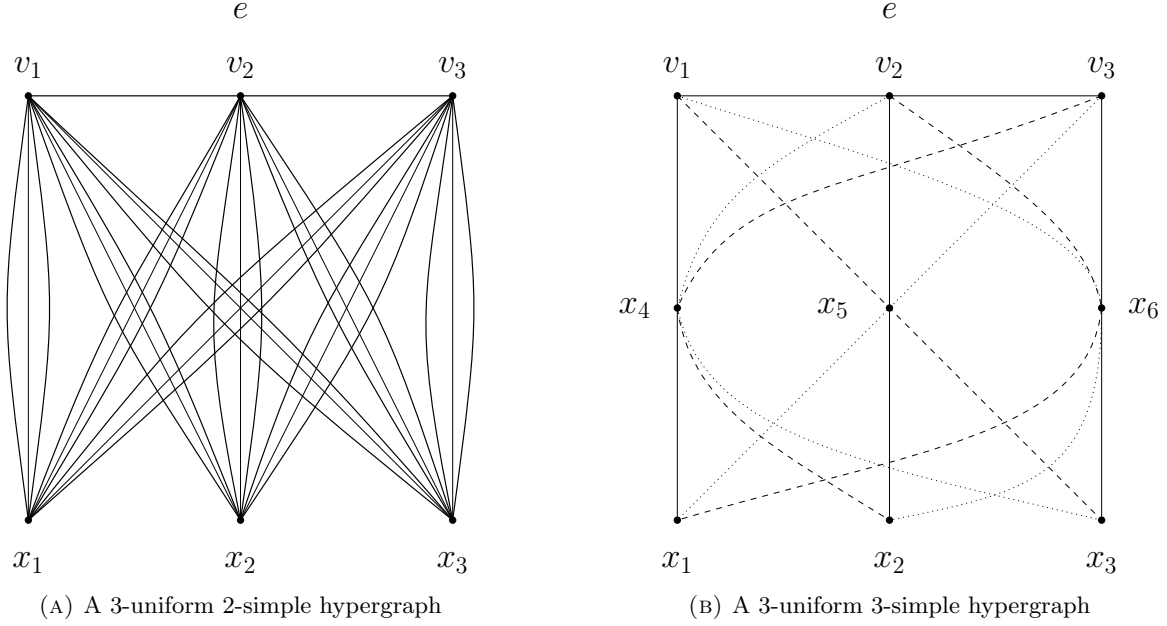

A similar intuition to the 2-simple case leads us to suspect that $b(e)$ will be maximized for a 3-uniform 3-simple hypergraph $\cH$ and edge $e$ where $\deg_{L(\cH)}(e)$ is large and, subject to this, $P_{L(\cH)}(e)$ is small. As we show in the next lemma, the configuration in Figure~\ref{fig: 3,3 optimum edge} realizes the bound in Lemma~\ref{lem: 3,3 a parameter}, up to a vanishing error.

\begin{lemma}\label{lem: extreme kk and 33}
    For every $D \in \mathbb Z^+$ and $k \in \mathbb Z^+$ for which there exists an affine plane of order $k$, there exists a $k$-uniform $k$-simple hypergraph $\cH$ of maximum degree at most $D$ and an edge $e \in \cH$ such that
  \begin{equation*}
    \theb_{L(\cH), kD}(e)= 1-\frac{k-1}{2k^2} + \frac{(k-1)(k-2)}{6k^4} + o(1),
  \end{equation*}
  where $o(1)$ goes to 0 as $D$ goes to infinity.
  
  In particular, for every $D \in \mathbb Z^+$, there exists a $3$-uniform $3$-simple hypergraph $\cH$ of maximum degree at most $D$ and an edge $e \in \cH$ such that
    \begin{equation*}
    b_{L(\cH),3D}(e) \geq 1 - \frac{1}{9} + \frac{1}{3^5}-o(1).
  \end{equation*}
\end{lemma}

\begin{proof}
  Let $D-k < D' \le D$ such that $D'-1$ is divisible by $k$, and suppose that there exists an affine plane of order $k$. 
  Choose a line to play the role of $e$, and let $\mathcal L$ be the set of lines not parallel to $e$.

  To form $\cH$, take the points of the affine plane as the vertices of $\cH$, include $e$ as an edge, and include $(D' - 1)/k$ copies of each line in $\mathcal L$ as edges. Each point of $e$ lies on one line from each of the $k$ parallel classes represented in $\mathcal L$ and thus has degree at most $D' \leq D$, while every other point has degree at most $D' - 1$. Hence, $\cH$ has maximum degree at most $D$, as desired, and the edge $e$ intersects exactly $kD' - k$ other edges.
  Moreover, two neighbors of $e$ are disjoint precisely when the two corresponding lines in the affine plane are parallel, so we have
  \begin{equation*}
    P_{L(\cH)}(e) = k \binom{k}{2} \left(\frac{D'-1}{k}\right)^2 = \frac{k-1}{2}D^2 + o(D^2).
  \end{equation*}
  Similarly, three neighbors of $e$ are pairwise disjoint when the three corresponding lines are parallel, so we have
  \begin{equation*}
    T_{L(\cH)}(e) = k\binom{k}{3}\left(\frac{D'-1}{k}\right)^3 = \frac{(k-1)(k-2)}{6k} D^3 + o(D^3).
  \end{equation*}
  Altogether, we have
  \begin{equation*}
    \theb_{L(\cH), kD}(e)= 1-\frac{k-1}{2k^2} + \frac{(k-1)(k-2)}{6k^4} + o(1),
  \end{equation*}
  as desired.

  It is well known that affine planes of order $k$ exist whenever $k$ is a prime power, so we get the desired result for a 3-uniform hypergraph if we let $k=3$, in which case the construction matches Figure \ref{fig: 3,3 optimum edge}.
\end{proof}


\subsection{Structural lemmas and tools}

To prove the bounds in Lemmas \ref{lem: 3,2 a parameter}, \ref{lem: 3,3 a parameter}, and \ref{lem: kk a parameter}, we analyze a hypergraph $\cF$ and an edge $f$ which maximize $\theb(f)$. The following lemma shows that such a maximizing pair $(\cF, f)$ must have $f$ intersecting as many edges as possible -- to wit, $\deg_{L(\cF)}(f)$ as large as possible. 

\begin{lemma}\label{lem: structure1}
    Let $k,t\in \mathbb{Z}^+$ and $D$ sufficiently large. Suppose $\mathcal{F} \in H(k,t,D)$ and $f \in E(\mathcal{F})$ such that $$b_{L(\mathcal{F}), kD}(f) = \max_{\cH \in H(k,t,D)}\max_{e\in E(\cH)} b_{L(\cH), kD}(e).$$
Then $\deg_{\mathcal{F}}(v) = D$ for every $v \in f$, and $f$ intersects exactly $kD-k$ other edges.
\end{lemma}

\begin{proof}[Proof of Lemma \ref{lem: structure1}]
    Label the vertices of $f$ as $v_1, v_2, \dots, v_k$. Suppose first, by way of contradiction, that some vertex of $f$ has degree less than $D$; without loss of generality, assume $\deg_{\cF}(v_1) < D$. Let $x_1, \dots, x_{k-1}$ be new vertices, and define a new hypergraph $\cF_1 \in H(k,t,D)$ by
    \[
      V(\cF_1) := V(\cF) \cup \{x_1,\dots,x_{k-1}\}
      \quad\text{and}\quad
      E(\cF_1) := E(\cF) \cup \{\{v_1,x_1,\dots,x_{k-1}\}\}.
    \]
    Let $f_1$ denote the edge $\{v_1,\dots,v_k\}$ in $\cF_1$. Now, we compare $f$ in the original hypergraph with $f_1$ in the new hypergraph.
    \begin{itemize}
        \item $\deg_{L(\cF_1)}(f_1) = \deg_{L(\cF)}(f) +1 $
        \item $P_{L(\cF_1)}(f_1) = P_{L(\cF)}(f) +\deg_{L(\cF)}(f) - (\deg_{\cF}(v_1)-1) $
        \item $T_{L(\cF_1)}(f_1) = T_{L(\cF)}(f) + P_{L(\cF)}(f)  - (\# \text{independent pairs involving edges containing } v_1) $
    \end{itemize}
    Combining these equations yields:
    \begin{align*}
        \theb_{L(\cF_1),kD}(f_1) &= \theb_{L(\cF),kD}(f) + \frac{1}{kD} - \frac{\deg_{L(\cF)}(f) - \deg_{\cF}(v_1)+1}{(kD)^2} \\ & \phantom{space it out}+ \frac{P_{L(\cF)}(f)-(\# \text{independent pairs involving edges containing } v_1)}{(kD)^3} \\
        &\ge \theb_{L(\cF),kD}(f) + \frac{1}{kD} - \frac{\deg_{L(\cF)}(f) +1}{(kD)^2}.
    \end{align*}
    Note that the number of independent pairs in $N_{L(\cF)}(f)$ involving edges containing $v_1$ is at most the total number of independent pairs, $P_{L(\cF)}(f)$, which is why we can drop the last term. Since $\deg_{L(\cF)}(f) < kD$, the above inequality implies $\theb_{L(\cF_1),kD}(f_1) > \theb_{L(\cF),kD}(f)$, which contradicts the choice of $(\cF, f)$. 

    The second part of the lemma is proved similarly: if $(\cF,f)$ are such that $f$ intersects fewer than $kD-k$ edges, then we can adjust to get a new pair $(\cF_2,f_2)$ with a parameter $\theb$ at least as large. In particular, if $(\cF,f)$ satisfies $\deg_{\cF}(v_i) = D$ but $f$ intersects fewer than $kD-k$ edges, then there must be an edge $e\in E(\cF)$ such that $|e\cap f| \ge 2$. Without loss of generality, assume $v_1,v_2 \in e$. Define $\cF_2 \in H(k,t,D)$ where $V(\cF_2) := V(\cF) \cup \{x_1,...,x_{k-1}\}$ and $E(\cF_2) := (E(\cF) \cup \{e_1 , e_2 \}) \setminus \{e\}$ where $e_1:= \{v_1,x_1,...,x_{k-1}\}$ and $e_2:=(e\cup\{x_1\})\setminus\{v_1\}$.

    Label the edge $\{v_1,...,v_k\} \in E(\cF_2)$ as $f_2$. Now, we compare $f$ from the original hypergraph to our new edge $f_2$. As before, $\deg_{L(\cF_2)}(f_2) = \deg_{L(\cF)}(f) +1 $. And since $T_{L(\cF_2)}(f_2) \ge T_{L(\cF)}(f)$ and $P_{L(\cF_2)}(f_2) < P_{L(\cF)}(f) +kD $, overall $\theb_{L(\cF_2),kD}(f_2) > \theb_{L(\cF),kD}(f)$. This contradicts the choice of $(\cF, f)$.
\end{proof}

Given a hypergraph $\cH$ and an edge $e\in E(\cH)$, let $X(e,\cH)\subseteq V(\cH)$ be the set of vertices outside $e$ that appear in some edge intersecting $e$. That is, $$X(e,\cH):= \bigcup \{h\in E(\cH): h\cap e\ne \emptyset\} \setminus e .$$

For vertices $x \in X(e,\cH)$ and $v\in e$, let $\deg_{\cH}(v,x)$ denote the number of edges of $\cH$ containing both $v$ and $x$.

The previous lemma, Lemma \ref{lem: structure1}, argued that a pair $(\cF,f)$ which maximizes the parameter $b(f)$ must have $\deg_{L(\cF)}(f) =kD-k$. This next lemma expands on that property to give some explicit structural information about such a pair $(\cF,f)$.

\begin{lemma}\label{lem: tableop new}
    Let $k,D \in \mathbb{N}$, let $\cH$ be a $k$-uniform hypergraph of maximum degree at most $D$, and let $e\in E(\cH)$. If $e$ intersects $kD-k$ edges of $\cH$, then the following holds:
    \begin{enumerate} [label=\textup{(S\arabic*)}]
        \item \label{case s1 new} $|f\cap e|\le 1$ for every $f\in E(\cH)\setminus \{e\}$;
        
        \item \label{case s2 new} $\sum_{v\in e} \deg_{\cH}(v,x) \le D$ for all $x\in X(e,\cH)$;

        \item \label{case s3 new} $\sum_{x\in X(e,\cH)} \deg_{\cH}(v,x) = (k-1)(D-1)$ for all $v\in e$. 
\end{enumerate}
\end{lemma}

\begin{proof}[Proof of Lemma \ref{lem: tableop new}]
    Label the vertices of $e$ as $v_1,v_2,...,v_k$, and label the vertices of $X(e,\cH)$ as $x_1,x_2,...,x_n$.
    
    First, we prove \ref{case s1 new} by contradiction. Suppose there exists $f\in E(\cH)\setminus \{e\}$ such that $|f\cap e|>1$. Without loss of generality, suppose $v_1,v_2\in f$. Let us count all the edges which intersect $e$. Because $\deg_{\cH}(v_i) \le D$, each $v_i$ is contained in at most $D-1$ edges which intersect $e$. Thus, the total number of edges which intersect $e$ is at most $k(D-1)=kD-k$. But since edge $f$ is counted twice (once for $v_1$ and for $v_2$), the total number of edges is actually at most $k(D-1)-1$, which is too small, a contradiction.

    Statement \ref{case s2 new} is also proved using the degree restriction. Because of \ref{case s1 new}, the set of edges containing $\{v_i,x_j\}$ is disjoint from the set of edges containing $\{v_i',x_j\}$, where $i,i'\in [k]$ are distinct. Thus $\sum_{i=1}^k \deg_{\cH}(v_i,x_j) \le \deg_{\cH}(x_j)\le D$.
    
    Let us now prove statement \ref{case s3 new}. Because $e$ intersects $kD-k$ edges, and because of \ref{case s1 new}, each $v_i$ is contained in $D-1$ edges intersecting $e$ (not including $e$ itself), and each of these edges contain exactly $k-1$ vertices from $X(e, \cH)$. 
    One such edge contributes $k-1$ to the sum $\sum_{j=1}^n \deg_{\cH}(v_i,x_j)$, and since there are $D-1$ such edges, we have $\sum_{j=1}^n \deg_{\cH}(v_i,x_j) = (k-1)(D-1)$.
\end{proof}

For an edge $e = \{v_1, v_2, \dots, v_k\}$ in a hypergraph $\cH$, for every $x\in X(e,\cH)$, define the \textit{weight} of $x$, denoted by $w_{\cH, e}(x)$, to be the value $$ w_{\cH, e}(x) := \sum_{1\le i< i'\le k} \deg_{\cH}(v_i,x) \cdot \deg_{\cH}(v_{i'},x). $$

For any set of vertices $X\subseteq V(\cH)$, define its \textit{weight} to be $w_{\cH, e}(X):= \sum_{x\in X} w_{\cH, e}(x)$. We will drop the subscripts $\cH$ and $e$ when these are clear from context. We also define the real-valued parameter $Y_{\cH}(e)$ as follows: $$Y_{\cH}(e) := \sum_{1\le i < j \le k } \sum_{e_1 \ni v_i, e_2\ni v_j} \max \{ 0, |e_1\cap e_2|-1\},$$
where the second sum is taken over edges $e_1,e_2\neq e$.

We define these two values, $w_{\cH,e}(X)$ and $Y_{\cH}(e)$, because they are useful for determining the number of edges in the line graph of $\cH$ induced on $N_{L(\cH)}(e)$. This, in turn, is useful for determining the number of independent pairs, which is used to calculate the parameter $b(f)$. The exact relationship using $w_{\cH,e}(X)$ and $Y_{\cH}(e)$ is given in this next lemma.

\begin{lemma}\label{lem: tableop new S4}
    Let $k,D \in \mathbb{N}$, let $\cH$ be a $k$-uniform hypergraph of maximum degree at most $D$, and let $e \in E(\cH)$, and let $G$ be the line graph of $\cH$ induced on $N_{L(\cH)}(e)$.
    If $e$ intersects $kD-k$ edges of $\cH$, then \label{case s4}
    $$ |E(G)| = k \binom{D-1}{2} + w_{\cH, e}(X) - Y_{\cH}(e). $$
\end{lemma}

\begin{proof}[Proof of Lemma \ref{lem: tableop new S4}]

    Label the vertices of $e$ as $v_1,v_2,...,v_k$. Also, let $X$ denote $X(e,\cH)$, and label its elements as $x_1,x_2,...,x_n$.

Let us consider the structure of $G$. Because $e$ intersects $kD-k$ edges, and since $\cH$ is $k$-uniform, necessarily $\deg_{\cH}(v_i) = D$ for each $i\in [k]$. Because of this, $G$ consists of $k$ cliques of size $D-1$, one for each $v_i$, with some number of additional edges between them. There are $k\binom{D-1}{2}$ edges coming from within these cliques.

    All other edges in $G$ go between cliques. Suppose there is an edge occurring between the clique associated with $v_i$ and that of $v_{i'}$. This happens if and only if there is some $x_j\in X$ and two edges $e_1,e_2$ in $\cH$ such that $\{v_i,x_j\} \subset e_1$ and $\{v_{i'},x_j\} \subset e_2$. Thus, $\{e_1,e_2\}$ is one edge in $G$, and this pair also contributes one to the value of $w_{\cH,e}(x_j)$. However, if $e_1,e_2$ overlap on multiple vertices from $X$, then it will be overcounted. 

    More explicitly, the number of edges in $G$ can be calculated as follows:
    $$|E(G)| = k\binom{D-1}{2} + \sum_{1\le i < j \le k} \sum_{e_1 \ni v_i, e_2\ni v_j} \mathbbm{1}[e_1\cap e_2\ne \emptyset].$$
    Additionally,
\begin{align*}
  \sum_{1 \leq i < j \leq k}\sum_{e_1 \ni v_i, e_2 \ni v_j}\mathbbm{1}[e_1 \cap e_2 \neq \emptyset]
  &= \sum_{1 \leq i < j \leq k}\sum_{e_1 \ni v_i, e_2 \ni v_j}\min\{1, |e_1 \cap e_2|\}\\
  &= w_{\cH,e}(X) - \sum_{1 \leq i < j \leq k}\sum_{e_1 \ni v_i, e_2 \ni v_j}\max\{0, |e_1 \cap e_2| - 1\}.
\end{align*}
Combining the two equalities yields the desired result.
\end{proof}


\subsection{Proof of Lemma~\ref{lem: 3,2 a parameter}}\label{subsection: 3,2 a param proof}

  This subsection is devoted to the proof of Lemma~\ref{lem: 3,2 a parameter}. To that end, throughout this subsection we assume $D$ is an integer sufficiently large to satisfy various inequalities, and we let $\cF \in H(3,2,D)$ and $f \in E(\cF)$ such that
  \begin{equation}
    \label{eq:3-uniform-2-simple-maxF}
    b_{L(\cF),3D}(f) = \max_{\cH\in H(3,2,D)}\max_{e\in E(\cH)} b_{L(\cH),3D}(e).
  \end{equation}

  From Lemma~\ref{lem: structure1}, we know that $f$ intersects $3(D-1)$ edges and each vertex of $f$ is in $(D-1)$ edges other than $f$.

  Let $X \coloneqq X(f, \cF)$.
  Let $G$ be the subgraph of the line graph of $\cF$ induced on $N_{L(\cF)}(f)$, and note that $P_{L(\cF)}(f)$ and $T_{L(\cF)}(f)$ are the number of edges and triangles in $\overline G$, respectively (we will omit the subscripts moving forward). Since $\binom{3D - 3}{2} - 3\binom{D - 1}{2} = 3D^2 - 6D + 3$ and $Y_{\cF}(f) \geq 0$, by Lemma~\ref{lem: tableop new S4} with $\overline G$ playing the role of $G$, we have
  \begin{equation}
    \label{eq:3-uniform-2-simple-pairs-bound}
    P(f) \geq 3D^2 - 6D + 3 - w_{\cF,f}(X).
  \end{equation}

	  Let $X_s \coloneqq \{x \in X : \sum_{v\in f}\deg_{\cF}(v,x) < D^{2/3}\}$, and let $X_b \coloneqq X \setminus X_s$.

    The vertices in $X_s$ have small total degree into $f$, so their contribution to $w(X)$ is negligible; the next two lemmas make this precise and control the degrees from $f$ into $X_b$.

	  \begin{lemma}\label{lem:degree-to-big-bound}
    For every $v \in f$,
    \begin{equation*}
      \sum_{x \in X_b}\deg(v, x) \leq D + 18D^{2/3}.
    \end{equation*}
  \end{lemma}
  \begin{proof}
  Since $\cF$ has maximum degree at most $D$, by definition of $X_b$, we have
  \begin{equation*}
    |X_b|\cdot D^{2/3} \leq \sum_{x \in X_b}\sum_{v\in f}\deg_{\cF}(v,x) \leq \sum_{x \in X}\sum_{v\in f}\deg_{\cF}(v,x) = 6(D - 1),
  \end{equation*}
  so
  \begin{equation*}
    \label{eq:3-uniform-2-simple-Xb-bound}
    |X_b| \leq 6D^{1/3}.
  \end{equation*}
  Let $v \in f$.
  Since $\cF$ is $2$-simple, there are at most $\binom{|X_b|}{2}$ edges in $\cF$ of the form $\{v,x,x'\}$ where $x,x' \in X_b$, so by the previous inequality,
  \begin{equation*}
    \sum_{x \in X_s}\deg_{\cF}(v, x) \geq D - 1 - \binom{|X_b|}{2} \geq D - 18D^{2/3}.
  \end{equation*}
  Hence, since $\sum_{x \in X_b}\deg(v, x) = \sum_{x \in X}\deg(v, x) - \sum_{x \in X_s}\deg(v, x) = 2(D-1) - \sum_{x \in X_s}\deg(v, x)$, by the previous inequality, we have $\sum_{x \in X_b}\deg(v,x) \leq D + 18D^{2/3}$, as desired.
  \end{proof}

  \begin{lemma}\label{lem:3uniform-2simple-pairs-bound}
    $P(f) \geq 3D^2 - 6D + 3 - w(X_b) - 2D^{5/3}$.
  \end{lemma}
  \begin{proof}
    By Lemma~\ref{lem: tableops redo} with $k=3$, $n=|X_s|$, $C = D^{2/3}$, and $6(D - 1) = T$, we have
    \begin{equation*}
      w(X_s) \leq \frac{2}{6}\left(6(D - 1)D^{2/3}\right) \leq 2D^{5/3}.
    \end{equation*}
    Hence, since $w(X) = w(X_b) + w(X_s)$, the inequality follows from \eqref{eq:3-uniform-2-simple-pairs-bound}.
  \end{proof}

  We next bound the number of independent triples by separating off the contribution from edges using vertices outside the three vertices of $X_b$ with largest total degree into $f$. To that end, label the elements of $X_b$ as $x_1, \dots, x_n$, where $\sum_{v\in f}\deg_{\cF}(v,x_i) \geq \sum_{v\in f}\deg_{\cF}(v,x_{i+1})$ for each $i \in [n-1]$, and let $f = \{v_1,v_2,v_3\}$.
	  
	  \begin{lemma}\label{lem:3uniform-2-simple-triples-bound}
    \begin{equation*}
      T(f) \leq \sum_{(i,j,k)\in S_3}\deg(v_i, x_1)\cdot \deg(v_j, x_2)\cdot \deg(v_k, x_3) + 3D^3 + 6D^2 - D^2\sum_{i=1}^3\sum_{j=1}^3\deg(v_i, x_j),
    \end{equation*}
    where $S_3$ is the symmetric group of $3$ elements.
  \end{lemma}
  \begin{proof}

    Partition $N_{L(\cF)}(f)$ into $V_1 \cup V_2 \cup V_3$, as follows. Let $V_1$ be the set of edges $g \in E(\cF)$ such that $|g \cap (X\setminus\{x_1,x_2,x_3\})| = 2$, let $V_2$ be the set of edges $g \in E(\cF)$ such that $|g \cap \{x_1,x_2,x_3\}| = 2$, and let $V_3 = N_{L(\cF)}(f) \setminus (V_1 \cup V_2)$. Since every $g \in N_{L(\cF)}(f)$ intersects $f$ in one vertex, we have $|g \cap \{x_1,x_2,x_3\}| = |g \cap (X \setminus \{x_1,x_2,x_3\})| = 1$ for every $g \in V_3$.

    Since $f$ intersects $3(D - 1)$ edges, we have $|V_1| + |V_2| + |V_3| = 3(D - 1)$, and $\sum_{i=1}^3\sum_{j=1}^3 \deg(v_i, x_j) = 2|V_2| + |V_3|$.
    Hence, 
    \begin{equation}\label{eq:3uniform-2simple-V1bound}
      |V_1| = 3(D-1) + |V_2| - \left(\sum_{i=1}^3\sum_{j=1}^3 \deg(v_i, x_j)\right) \leq 3D + 6 - \left(\sum_{i=1}^3\sum_{j=1}^3 \deg(v_i, x_j)\right),
    \end{equation}
    where in the last inequality we used $|V_2| \leq 9$, which follows because $\cF$ is $2$-simple.

    The number of triangles of $\overline G$ involving at least one vertex of $V_1$ is at most $|V_1|D^2$.
    Since $\cF$ is $2$-simple, every triangle containing some vertex in $V_2$ also has at least one vertex in $V_1$.
    Hence, the remaining triangles of $\overline G$ have vertices $g_1, g_2, g_3 \in V_3$, such that $x_i,v_1 \in g_1$, $x_j,v_2\in g_2$, and $x_k,v_3 \in g_3$ for distinct $i,j,k\in[3]$.
    Therefore, the number of triangles in $\overline G$ is at most
    \begin{equation*}
      T(f) \leq \sum_{(i,j,k)\in S_3}\deg(v_i, x_1)\cdot \deg(v_j, x_2)\cdot \deg(v_k, x_3) + D^2|V_1|,
    \end{equation*}
    and the result follows by combining the previous inequality with \eqref{eq:3uniform-2simple-V1bound}.
	  \end{proof}

  We now combine the pair and triple estimates above and finish by applying Lemma~\ref{lem:3xn-matrix-bound} to the resulting expression.

	  \begin{proof}[Proof of Lemma~\ref{lem: 3,2 a parameter}]
    By Lemmas~\ref{lem:3uniform-2simple-pairs-bound} and \ref{lem:3uniform-2-simple-triples-bound}, we have
    \begin{multline*}
      b_{L(\cF), 3D}(f) \leq 1 - \left(\frac{3D^2}{(3D)^2}\right) + \left(\frac{w(X_b)}{(3D)^2}\right) + \left(\frac{2D^{5/3} + 6D}{(3D)^2}\right) + \left(\frac{3D^3 + 6D^2}{(3D)^3}\right) +\\
      \left(\frac{\sum_{(i,j,k)\in S_3}\deg(v_i, x_1)\cdot \deg(v_j, x_2)\cdot \deg(v_k, x_3)}{(3D)^3}\right) - \left(\frac{D^2\sum_{i=1}^3\sum_{j=1}^3 \deg(v_i, x_j)}{(3D)^3}\right).
    \end{multline*}
    Note that
    \begin{equation*}
      1 - \left(\frac{3D^2}{(3D)^2}\right) + \left(\frac{2D^{5/3} + 6D}{(3D)^2}\right) + \left(\frac{3D^3 + 6D^2}{(3D)^3}\right) \leq 1 - \frac{2}{9} + o(1),
    \end{equation*}
    \begin{equation*}
      \frac{w(X_b)}{(3D)^2} = \frac{1}{9}\sum_{j=1}^n\sum_{1 \leq i < i' \leq 3}\left(\frac{\deg(v_i,x_j)}{D}\frac{\deg(v_{i'},x_j)}{D}\right),
    \end{equation*}
    \begin{equation*}
      \frac{\sum_{(i,j,k)\in S_3}\deg(v_i, x_1)\cdot \deg(v_j, x_2)\cdot \deg(v_k, x_3)}{(3D)^3} = \frac{1}{27}\sum_{(i,j,k)\in S_3}\left(\frac{\deg(v_i,x_1)}{D}\frac{\deg(v_j,x_2)}{D}\frac{\deg(v_k,x_3)}{D}\right),
    \end{equation*}
    and
    \begin{equation*}
      \frac{D^2\sum_{i=1}^3\sum_{j=1}^3 \deg(v_i, x_j)}{(3D)^3}
      =  \frac{1}{27}\sum_{i=1}^3\sum_{j=1}^3\left(\frac{\deg(v_i, x_j)}{D}\right).
    \end{equation*}

  We apply Lemma~\ref{lem:3xn-matrix-bound} to the matrix $A = (a_{i,j})_{i\in[3],j\in[n]}$ where $a_{i,j} \coloneqq \deg(v_i, x_j) / (D + 18 D^{2/3})$. We have $\sum_{i=1}^3 a_{i,j} \leq \deg(x_j) / (D + 18 D^{2/3}) < 1$ for each $j \in [n]$, and by Lemma~\ref{lem:degree-to-big-bound}, we have $\sum_{i=1}^3\sum_{j=1}^n a_{i,j} \leq 3$. Therefore, combining the inequalities above with Lemma~\ref{lem:3xn-matrix-bound}, we have
  \begin{equation*}
    b_{L(\cF),3D}(f) \leq 1 - \frac{2}{9} + o(1) + \frac{D + 18D^{2/3}}{D}\left(\frac{2}{3^5}\right) = 1 - \frac{2}{9} + \frac{2}{3^5} + o(1),
  \end{equation*}
  as desired.
\end{proof}

\subsection{Proof of Lemma~\ref{lem: 3,3 a parameter}}\label{subsection: 3,3 param proof}
This subsection is devoted to the proof of Lemma~\ref{lem: 3,3 a parameter}. To that end, throughout this subsection we assume $D$ is an integer sufficiently large to satisfy various inequalities, and we let $\cF \in H(3,3,D)$ and $f \in E(\cF)$ such that
  \begin{equation}
    \label{eq:3-uniform-3-simple-maxF}
    b_{L(\cF),3D}(f) = \max_{\cH\in H(3,3,D)}\max_{e\in E(\cH)} b_{L(\cH),3D}(e).
  \end{equation}

  From Lemma~\ref{lem: structure1}, we know that $f$ intersects $3(D-1)$ edges and each vertex of $f$ is in $(D-1)$ edges other than $f$.

  Let $X \coloneqq X(f, \cF)$ and $Y \coloneqq Y_{\cF}(f)$.
  Let $G$ be the line graph of $\cF$ induced on $N_{L(\cF)}(f)$, and note that $P(f)$ and $T(f)$ are the number of edges and triangles in $\overline G$, respectively. Since $\binom{3D - 3}{2} - 3\binom{D - 1}{2} = 3D^2 - 6D + 3$, by Lemma~\ref{lem: tableop new S4} with $\overline G$ playing the role of $G$, we have
  \begin{equation}
    \label{eq:3-uniform-3-simple-pairs-bound}
    P(f) = 3D^2 - 6D + 3 - w(X) + Y.
  \end{equation}

  The next two lemmas show that, for this maximizing pair, $P(f)$ is close to $D^2-9D$; this will let us write $P(f)=(1+\delta)D^2-9D$ with $\delta$ small.

\begin{lemma} \label{lem: 33pairs}
    $P(f) \geq D^2 - 9D + Y$.
\end{lemma}
  \begin{proof}
    Apply Lemma \ref{lem: tableops redo} with $k=3$, $n=|X|$, $C=D$, and $T=6(D-1)$ to get 
    \begin{align*}
        w(X) & \le \frac13 (6(D-1))(D) \le 2D^2.
    \end{align*}
    The result follows by combining the inequality above with \eqref{eq:3-uniform-3-simple-pairs-bound}.
  \end{proof}

\begin{lemma} \label{lem: delta 1.031}
    $P(f) \le (1.031)D^2-9D$.
\end{lemma}

\begin{proof}
    First, recall that $T(f)$ is the number of triangles in $\overline{G}$. Hence, by Lemma \ref{lem: trianglecount1} with $\overline{G}$ and $P(f)$ playing the roles of $G$ and $m$, respectively, we have
    \begin{equation*}
      T(f) \leq \left(\frac{P(f)}{3}\right)^{3/2}.
    \end{equation*}
    Therefore, if $P(f) > (1.031)D^2 - 9D$, then for sufficiently large $D$, 
    \begin{align*}
      b_{L(\cF),3D}(f) \leq \frac{3(D-1)}{3D} - \frac{P(f)}{(3D)^2} + \frac{({P(f)}/{3})^{3/2}}{(3D)^3} \le 0.89291.
    \end{align*}
    However, by Lemma \ref{lem: extreme kk and 33},
    \begin{equation*}
      \max_{\cH,e}b_{L(\cH), 3D}(e) \geq 1-\frac19 + \frac{1}{3^5} - o(1) \ge 0.893,
    \end{equation*}
    contradicting \eqref{eq:3-uniform-3-simple-maxF}. 
\end{proof}

Let $X_s$ be the maximal set of vertices in $X$ of lowest degree into $f$ such that
\begin{equation*}
  \sum_{v\in f}\sum_{x \in X_s} \deg_{\cF}(v,x) \leq D/2
\end{equation*}
(i.e.\ such that $\sum_{v\in f}\deg_\cF(v,x) \leq \sum_{v\in f}\deg_\cF(v,x')$ whenever $x \in X_s$ and $x'\in X \setminus X_s$), and let $X_b \coloneqq X \setminus X_s$.
The next lemma shows that, when $P(f)$ is close to $D^2-9D$, this low-degree part is small both in total degree into $f$ and in weight.

	  \begin{lemma} \label{lem: small Xs}
    For every $\delta \ge 0$, if $P(f) \leq (1 + \delta)D^2 - 9D$, then
    \begin{equation*}
      \sum_{v \in f}\sum_{x \in X_s}\deg_{\cF}(v, x) \leq \frac{8}{5}\delta D,
    \end{equation*}
    and
    \begin{equation*}
      w(X_s) \leq \frac{2}{75}\delta D^2.
    \end{equation*}
  \end{lemma}
  \begin{proof}
    Since the conclusions become weaker as $\delta$ increases, we may assume that $P(f) = (1+\delta)D^2 - 9D$.
    Let $\eps = \sum_{v \in f}\sum_{x \in X_s}\deg(v, x) / D$. By definition, $\eps \leq 1/2$. First, we will show $\eps \leq 8\delta/5$.

    By Lemma \ref{lem: tableops redo} with $k = 3$, $n = |X_b|$, $T = (6 - \eps)D$, and $C = D$, we have
    \begin{equation*}
      w(X_b) \leq \frac{1}{3}\left(5D^2 + (1 - \eps)^2D^2\right),
    \end{equation*}
    and by Lemma \ref{lem: tableops redo} with $\eps D$ playing the role of both $C$ and $T$, we have
    \begin{equation}\label{eq:3uniform-3simple-smallX-weight}
      w(X_s) \leq \frac{1}{3}\left(\eps D\right)^2.
    \end{equation}
    Therefore,
    \begin{equation*}
      w(X)  = w(X_b) + w(X_s) \leq \left(2 - \frac{2}{3}\eps + \frac{2}{3}\eps^2\right)D^2.
    \end{equation*}
    By \eqref{eq:3-uniform-3-simple-pairs-bound},
    \begin{equation*}
      w(X) \geq 3D^2 - 6D + 3 - P(f) \geq (2 - \delta)D^2.
    \end{equation*}
    Combining the previous two inequalities, we have
    \begin{equation}\label{eq:delta1-quadratic-bound}
      \frac{2}{3}\eps - \frac{2}{3}\eps^2 \leq \delta,
    \end{equation}
    so either
    \begin{equation*}
      \eps \leq \frac{1}{2} - \frac{1}{2}\sqrt{1 - 6\delta} \qquad \text{or} \qquad \eps \geq \frac{1}{2} + \frac{1}{2}\sqrt{1 - 6\delta}.
    \end{equation*}
    Since $\eps \leq 1/2$ and $\delta \leq .031$ by Lemma \ref{lem: delta 1.031}, we have $\eps \leq 1/2 - (1/2)\sqrt{1 - 6\delta} \leq .049$. Therefore, again by \eqref{eq:delta1-quadratic-bound}, we have $(2\eps/3)(1 - .049) \leq \delta$, and thus $\eps \leq 8\delta/5$, as desired. Finally, by \eqref{eq:3uniform-3simple-smallX-weight}, since $\eps^2 \leq (8\delta/5)(.049) \leq 2\delta/25$, we have $w(X_s) \leq 2\delta D^2/75$, as desired.
	  \end{proof}

The next lemma recovers one of the visible features of the extremal construction in Figure~\ref{fig: 3,3 optimum edge}: after discarding the low-degree part $X_s$, exactly six vertices outside $f$ remain.

	    \begin{lemma}\label{lem: Xb=6}
	    $|X_b|=6$.
	    \end{lemma}

    \begin{proof}
        If $0 \le \delta\le 0.031$ is defined such that $P(f) = (1+\delta)D^2-9D$, then $w(X) \ge (2-\delta)D^2$ by~\eqref{eq:3-uniform-3-simple-pairs-bound}.
    
        Since $\sum_{v\in f}\sum_{x \in X_s} \deg_{\cF}(v,x) \le D/2$ by definition and $\sum_{v\in f}\sum_{x \in X} \deg_{\cF}(v,x) = 6(D-1)$ by \ref{case s3 new}, we have $\sum_{v\in f}\sum_{x \in X_b} \deg_{\cF}(v,x)  > 5D$.        
        This implies that $|X_b| \ge 6$.

        In order to show that $|X_b|\le 6$ as well, let us proceed by contradiction. Let $|X_b|\ge 7$. Let $x'\in X_b$ be the vertex minimizing the quantity $m:= \sum_{v\in f} \deg_{\cF}(v,x')$. Since $\sum_{v\in f}\sum_{x \in X_s} \deg_{\cF}(v,x) \leq \frac85 \delta D \le 0.0496 D$ by Lemma \ref{lem: small Xs}, we have $m \ge (D/2) - (0.0496D) > 0.45 D$ by the maximality of $X_s$. Also, since $|X_b| \ge 7$, we have $m\le (6(D-1))/|X_b| \le \frac67 D$ by our choice of $x'$.

        Apply Lemma \ref{lem: tableops redo} with $k=3$, $n=|X|-1$, $C=D$, and $T=6(D-1)-m$ to get $w(X\setminus \{x'\}) \le \frac13 (5D^2 + (D-m)^2) = 2D^2 - \frac23 mD + \frac{m^2}{3}$. Additionally, by Lemma \ref{lem: tableops redo} with $k=3$, $n=1$, $C=T=m$, we have $w(x') \le \frac13 m^2$. 
        
        Since $0.45 D \le m \le \frac67 D$, this implies $w(X) =w(X\setminus \{x'\})+w(x') \le 2D^2 - \frac23 m(D-m) 
        \le 2D^2 - \frac23 (0.45D)(\frac17 D) 
        \le (2-0.04)D^2$, a contradiction.
    \end{proof}

	    Partition the vertex set of $\overline{G}$ into $V_1\cup V_2$ as follows. Let $V_1$ be the set of edges $g\in E(\cF)$ such that $g\cap X_s\ne \emptyset$, and let $V_2 = V(\overline G)\setminus V_1$. Since every $g\in N_{L(\cF)}(f)$ intersects $f$ in one vertex, we have $|g\cap X_b|=2$ for every $g\in V_2$.
With $|X_b|=6$ in hand, near-extremality of $w(X_b)$ forces the following degree sum to be close to its balanced value for every edge in $V_2$.

	  \begin{lemma}\label{lem: pm epsilon 83deltaprime}
    For every $\delta \ge 0$, if $w(X_b) \ge (2 - \delta)D^2$, then every $g \in V_2$ satisfies
    \begin{equation*}
      \sum_{x \in g\setminus f} \sum_{v\in f\setminus g} \deg_{\cF}(v, x) = \left(\frac{4}{3} \pm \sqrt{\frac{8}{3}\delta}\right)D.
    \end{equation*}
  \end{lemma}
  \begin{proof}
     By way of contradiction, suppose not. 
     That is, suppose
     \[
       \left|\sum_{x \in g\setminus f} \sum_{v\in f\setminus g}
       \deg_{\cF}(v, x) - \frac43 D\right|
       > \sqrt{8\delta/3}D
     \]
     for some $g\in V_2$. Since $|X_b|=6$ by Lemma~\ref{lem: Xb=6}, the ``moreover" part of Lemma~\ref{lem: tableops redo}, with $3$, $D$, $6$, and $\sqrt{8\delta/3}D$ playing the role of $k$, $C$, $n$, and $\delta$, respectively, implies that
     \[
       w(X_b) < \frac13 6D^2 - \frac{3(8\delta/3)D^2}{4(2)}
       = (2-\delta)D^2,
     \]
     a contradiction.
  \end{proof}

	  Let $T_1$ denote the number of triangles in $\overline{G}$ containing at least one vertex from $V_1$. Let $T_2$ denote the number of remaining triangles in $\overline{G}$.
The term $T_1$ will be handled directly using the smallness of $X_s$, so the main remaining task is to bound $T_2$ using the degree control from Lemma~\ref{lem: pm epsilon 83deltaprime}.
	  
	  \begin{lemma} \label{lem: t2 count}
    For every $\eps > 0$, if $\sum_{x \in g\setminus f} \sum_{v\in f\setminus g} \deg_{\cF}(v, x) = (4/3 \pm \eps)D$ for every $g \in V_2$, then
    \begin{equation*}
        T_2 \le \frac{1}{12} \left( \frac{\frac83 (P(f) + w(X_s)) + (4+2\eps)Y}{D^2} - \frac43 + 3\eps^2 \right)D^3 + \frac43 D^2. 
    \end{equation*}
  \end{lemma}

  \begin{proof}
    By Lemma \ref{lem: trianglecount2},
    \begin{equation}\label{eq:T2-cherry-count}
      T_2 \leq \frac{1}{12}\sum_{g \in V_2} \deg_{\overline G}(g)^2.
    \end{equation}
    For each $g\in V_2$, letting
    \begin{equation*}
      d(g):=2D- \sum_{x\in g\setminus f} \sum_{v\in f\setminus g}\deg_{\cF}(v,x)
    \end{equation*}
    and
    \begin{equation*}
      y(g):= \left|\left\{g'\in V(\overline{G}): g'\cap f\neq g \cap f
\text{ and } |g\cap g'\cap X_b|=2\right\}\right|,
    \end{equation*}
    we have $\deg_{\overline G} (g) \le d(g) + y(g)$, so
    \begin{equation}\label{eq:T2-sum-squares}
      \sum_{g \in V_2} \deg_{\overline G}(g)^2 \leq \sum_{g \in V_2}d(g)^2 + \sum_{g\in V_2}d(g)y(g) + \sum_{g \in V_2}y(g)\deg_{{\overline G}}(g).
    \end{equation}

    Since $\deg_{{\overline G}}(g) \leq 2D$ for every $g \in V_2$ and $\sum_{g\in V_2}y(g) \leq 2Y$, we have
    \begin{equation}\label{eq:T2-sum-y-deg}
      \sum_{g \in V_2}y(g)\deg_{{\overline G}}(g) \leq 2D\sum_{g \in V_2}y(g) \leq 4DY.
    \end{equation}
    Similarly, since $d(g) \leq (2/3 + \eps)D$ for every $g \in V_2$, we have
    \begin{equation}\label{eq:T2-sum-dy}
      \sum_{g\in V_2}d(g)y(g) \leq \left(\frac{2}{3} + \eps\right)D\sum_{g\in V_2}y(g) \leq \left(\frac{4}{3} + 2\eps\right)DY.
    \end{equation}
    By Fact \ref{lem: sums} with $(2/3 - \eps)D$ and $(2/3 + \eps)D$ playing the roles of $a$ and $b$, respectively, we have
    \begin{equation}\label{eq:T2-sum-d-squared}
      \sum_{g \in V_2}d(g)^2 \leq \frac{4D}{3}\left(\sum_{g\in V_2}d(g)\right) - |V_2|\left(\frac{4}{9} - \eps^2\right)D^2,
    \end{equation}
    and 
    \begin{equation}\label{eq:T2-sum-d}
      \sum_{g \in V_2} d(g) = 2D|V_2| - \sum_{g\in V_2}\sum_{x \in g\setminus f} \sum_{v\in f\setminus g} \deg_{\cF}(v, x)
       \leq 2D|V_2| - 2w(X_b) + D|V_1|.
    \end{equation}
    Combining \eqref{eq:T2-sum-squares}, \eqref{eq:T2-sum-y-deg}, \eqref{eq:T2-sum-dy}, \eqref{eq:T2-sum-d-squared}, and \eqref{eq:T2-sum-d}, we have
    \begin{equation*}
      \sum_{g\in V_2} \deg_{\overline G}(g)^2 \leq \left(\frac{20}{9} + \eps^2\right)|V_2|D^2 - \frac{8D}{3}w(X_b) + \frac{4}{3}|V_1|D^2 + \left(\frac{16}{3} + 2\eps\right)DY.
    \end{equation*}
    Note that $|V_1| + |V_2| \leq 3D$ and ${4}/{3} < {20}/{9}+\eps^2$, and by \eqref{eq:3-uniform-3-simple-pairs-bound}, we have $w(X_b) = 3D^2 - 6D + 3 - P(f) - w(X_s) + Y$, so by the previous inequality,
    \begin{align*}
      \sum_{g\in V_2} \deg_{\overline G}(g)^2 &\leq \left(\frac{8}{3}\cdot\frac{P(f)}{D^2} -\frac{4}{3} + 3\eps^2\right)D^3 + \frac{8D}{3}(w(X_s) + 6D) + \left(4 + 2\eps\right)DY\\
      &\leq \left(\frac{\frac{8}{3}(P(f) + w(X_s)) + (4 + 2\eps)Y}{D^2} - \frac{4}{3} + 3\eps^2\right)D^3 + 16D^2\\,
    \end{align*}
    and the result follows by \eqref{eq:T2-cherry-count}.     
	  \end{proof} 

We now finish by writing $P(f)=(1+\delta)D^2-9D$, using the preceding lemmas to bound the two triangle contributions $T_1$ and $T_2$, and substituting the resulting estimates into the definition of $b$.
	    
	  \begin{proof}[Proof of Lemma \ref{lem: 3,3 a parameter}]
    Let $\delta \in \mathbb R$ such that $P(f) = (1 + \delta)D^2 - 9D$.
    By Lemma \ref{lem: 33pairs}, $\delta \geq Y/D^2 \geq 0$, and by Lemma \ref{lem: delta 1.031}, $\delta \leq .031$. By Lemma \ref{lem: small Xs}, $\sum_{v\in f}\sum_{x \in X_s}\deg_{\cF}(v,x) \leq 8\delta D/5$ and $w(X_s) \leq 2\delta D^2/75$.

    The number of triangles in $\overline{G}$ is $T_1 + T_2$, where again $T_1$ denotes the number of triangles containing at least one vertex from $V_1$.

    Note that
    \begin{equation*}
      T_1 \leq |V_1|D^2 \leq D^2 \sum_{v\in f}\sum_{x \in X_s}\deg_{\cF}(v,x) \leq \frac{8}{5}\delta D^3.
    \end{equation*}

    Define $\delta':= \frac{77}{75}\delta$. By \eqref{eq:3-uniform-3-simple-pairs-bound}, $w(X_b)=w(X)-w(X_s) \ge (2-\delta)D^2 - \frac{2}{75}\delta D^2 = (2-\delta')D^2$.
    
    Let $\eps = \sqrt{8\delta'/3}$, and note that $\eps \leq \sqrt{8(77/75)(.031)/3} \leq 1/2$. 
    By Lemma \ref{lem: pm epsilon 83deltaprime} with $\delta'$ playing the role of $\delta$, we have $\sum_{x \in g\setminus f}\sum_{v\in f\setminus g}\deg_{\cF}(v,x) = (4/3 \pm \eps)D$ for every $g\in V_2$. Hence, by Lemma \ref{lem: t2 count},
    \begin{align*}
      T_2
      &\le \frac{1}{12} \left( \frac{\frac83 (P(f) + w(X_s)) + (4+2\eps)Y}{D^2} - \frac43 + 3\eps^2 \right)D^3 + \frac{4}{3}D^2\\
      &\leq \frac{1}{12} \left( \frac{\frac83 ((1+\delta)D^2-9D + \frac{2}{75}\delta D^2) + (4+2\eps)\delta D^2}{D^2} - \frac43 + 3\eps^2 \right)D^3 + \frac{4}{3}D^2\\
      &\leq \frac{1}{12} \left( \frac83 \left((1+\delta) + \frac{2}{75}\delta\right) + 5\delta - \frac43 + 8\left(\frac{77}{75}\delta\right) \right)D^3 - \frac{2}{3}D^2\\
      &\leq \left(\frac{1}{9} + 1.33\delta\right)D^3,
    \end{align*}
    where in the second inequality, we used that $P(f) = (1+\delta)D^2-9D$, $w(X_s)\leq \frac{2}{75}\delta D^2$, and $Y \leq \delta D^2$, and in the third inequality, we used $3\eps^2 = 8\delta' = 8(77 \delta/75)$ and $\eps \leq 1/2$.

    Altogether,
    \begin{align*}
        T_1 + T_2 &\leq \frac{8}{5}\delta D^3 + \left(\frac{1}{9} + 1.33\delta\right)D^3 = \left(\frac{1}{9} + 2.93\delta\right)D^3 \leq \left(\frac{1}{9} + 3\delta\right)D^3.
    \end{align*}

    Finally, we compute
    \begin{equation*}
      b_{L(\cF),3D}(f) \leq 1 - \frac{(1 + \delta)D^2 - 9D}{(3D)^2} + \frac{(1/9 + 3\delta)D^3}{(3D)^3} = 1 - \frac{1}{9} + \frac{1}{3^5}+o(1),
    \end{equation*}
    as desired.
  \end{proof}

\subsection{Proof of Lemma~\ref{lem: kk a parameter}}\label{subsection: kk param proof}

Now we turn our attention to $k$-uniform $k$-simple hypergraphs. This first lemma gives a lower bound on $P(f)$ for a hypergraph $\cF$ which maximizes the parameter $\theb$.

\begin{lemma}\label{lem: kunif pairs}
    Let $k\in \mathbb{Z}^+$ and $D$ be sufficiently large.  
    Suppose $\mathcal{F}$ is a $k$-uniform hypergraph with maximum degree at most $D$, and $f \in E(\mathcal{F})$ such that $$\theb_{L(\mathcal{F}), kD}(f) = \max_{\cH \in H(k,k,D)}\max_{e\in E(\cH)} 
    \theb_{L(\cH), kD}(e).$$
    Then $$P_{ L(\cF)}(f) \ge \left( \frac{k-1}{2}\right)D^2 - k^2D.$$
\end{lemma}

\begin{proof}[Proof of Lemma \ref{lem: kunif pairs}]
    Label the vertices of $f$ as $v_1, v_2, \dots, v_k$, let $X:=X(f,\cF)$, and label the elements of $X$ as $x_1, x_2, \dots, x_n$. Let $G$ be the line graph of $\cF$ induced on $N_{L(\cF)}(f)$.

    From Lemma \ref{lem: structure1}, we know that $f$ intersects $kD-k$ edges. By Lemma \ref{lem: tableop new S4}, we can conclude that 
    \begin{align}\label{eq 3.19 E}
        |E(G)| \le k \binom{D-1}{2} + w(X).
    \end{align}

    By Lemma \ref{lem: tableop new}, we have $\sum_{i=1}^k \deg_{\cF}(v_i,x_j) \le D$ for all $j \in [n]$ and $\sum_{j=1}^n \deg_{\cF} (v_i,x_j) = (k-1)(D-1)$ for all $i\in [k]$. This implies that $\sum_{i=1}^k \sum _{j=1}^n \deg_{\cF} (v_i,x_j) = k(k-1)(D-1)$. 

    By Lemma \ref{lem: tableops redo}, with $k(k-1)D$ and $D$ playing the roles of $T$ and $C$, respectively, we have 
    \begin{align} \label{eq 3.19 w}
        w(X) \le \frac{k-1}{2k} (k(k-1)D)D= \frac{(k-1)^2}{2}D^2.
    \end{align}

    Using \eqref{eq 3.19 E} and \eqref{eq 3.19 w}, we have 
    \begin{multline*}
        P(f) = \binom{\deg_{L(\cF)}(f)}{2}- |E(G)| 
        \ge \binom{kD-k}{2} - k\binom{D-1}{2} - \frac{(k-1)^2}{2}D^2 \\
        =\frac{(k-1)}{2}D^2 - (k^2-k)D + \frac{k^2-k}{2}
        \ge \frac{(k-1)}{2}D^2 - (k^2)D, 
    \end{multline*}
    as desired. 
\end{proof}

Now we can prove our general bound on the $\theb$ parameter for any $k$-uniform hypergraph with large enough maximum degree.

\begin{proof}[Proof of Lemma \ref{lem: kk a parameter}]
Suppose $\cF \in H(k,k,D)$ and $f\in E(\cF)$ such that $$\theb_{L(\cF), kD} (f) = \max_{\cH \in H(k,k,D)}\max_{e\in E(\cH)} \theb_{L(\cH), kD}(e).$$
    
    By Lemma \ref{lem: structure1}, we have $\deg_{L(\cF)}(f) = kD-k$.
    By Lemma \ref{lem: kunif pairs}, we have $P(f) \ge (\frac{k-1}{2})D^2 - k^2D$. 
    Suppose that $P (f) = (\frac{k-1}{2}+\delta)D^2 - k^2D$ for some particular $\delta \ge 0$. Independent triples in $N_{L(\cF)}(f)$ are triangles in the complement of $N_{L(\cF)}(f)$, which is a $k$-partite graph with $P(f)$ edges. Thus, Lemma \ref{lem: trianglecount1} implies that $T(f) \le \binom{k}{3}(P(f) / \binom{k}{2})^{3/2}$, so
    \begin{align*}
        b_{L(\cF), kD} (f) &< \frac{kD-k}{kD} - \frac{P(f)}{(kD)^2} + \frac{\binom{k}{3}(P(f) / \binom{k}{2})^{3/2}}{(kD)^3}  \\
        & \le 1 - \frac{1}{k^2}\left(\frac{k-1}{2}+\delta\right) + \frac{\binom{k}{3}}{k^3 \binom{k}{2}^{3/2}}\left(\frac{k-1}{2}+\delta\right)^{3/2} \\
        & \le 1 - \frac{k-1}{2k^2} + \frac{(k-1)(k-2)}{6k^3 \sqrt{k}}.
    \end{align*}
    Therefore, for any $\cH \in H(k,k,D)$ and $e\in E(\cH)$, we must have  $\theb_{L(\cH), kD}(e) \le 1 - \frac{k-1}{2k^2} + \frac{(k-1)(k-2)}{6k^3 \sqrt{k}} + o(1)$, as desired.
\end{proof}

\subsection{Proofs of Theorems \ref{thm: 3,2 corr}, \ref{thm: 3,3 corr}, and \ref{thm: kk corr}}

With our lemmas upper bounding $b(e)$ in hand, we need only apply the main coloring result, Theorem \ref{thm:gct}, to each case.

\begin{proof}[Proof of Theorem \ref{thm: 3,2 corr}]
    Given $\iota>0$, let $\eps = \iota / 2$. Let $D_0$ be sufficiently large. 

    As in the theorem statement, let $\cH \in H(3,2,D)$ where $D\ge D_0$. Apply Theorem \ref{thm:gct} for the given $\eps$, and suppose we have chosen our $D_0$ to be at least as large as the $D_0$ from the theorem. Let us check that $\cH$ satisfies the hypothesis.

    Let $D' \ge \lfloor \eps D /3 \rfloor$ and let
    $H\subseteq \cH$ such that $\Delta(H) \le D'$. Then we can apply Lemma \ref{lem: 3,2 a parameter} to $H$, supposing that we have chosen $D_0$ large enough such that $D'$ is still large. The lemma gives $$b_{L(H), 3D'} (e) \le 1 - \frac29 + \frac{2}{3^5} + \frac{\iota}{2}$$ for all $e\in E(H)$.

    Thus, by Theorem \ref{thm:gct}, we have $\chi'(\cH) \le (1 - \frac29 + \frac{2}{3^5} + \frac{\iota}{2} + \eps)3D = (1 - \frac29 + \frac{2}{3^5} + \iota)3D$, as desired.
\end{proof}

\begin{proof}[Proof of Theorem \ref{thm: 3,3 corr}]
    Just as in the previous proof, let $\iota >0$ and $\eps = \iota/2$, and let $D_0$ be sufficiently large. 

    Let $\cH\in H(3,3,D)$ where $D\ge D_0$. For all $D'\ge \lfloor \eps D/3 \rfloor$ and $H\subset \cH$ with $\Delta(H) \le D'$, Lemma \ref{lem: 3,3 a parameter} implies that $b_{L(H),3D'}(e) \le 1 - \frac19 + \frac{1}{3^5} + \frac{\iota}{2}$ for every $e\in E(H)$. Thus, by Theorem \ref{thm:gct}, we have $\chi'(\cH) \le (1 - \frac{1}{9} + \frac{1}{3^5} + \iota)3D$, as desired.
\end{proof}

\begin{proof}[Proof of Theorem \ref{thm: kk corr}]
    Again, proceed as in the previous proof. Let $\iota >0$ and $\eps = \iota/2$, and let $D_0$ be sufficiently large. 

    Let $\cH\in H(k,k,D)$ where $D\ge D_0$. For all $D'\ge \lfloor \eps D/3 \rfloor$ and $H\subset \cH$ with $\Delta(H) \le D'$, Lemma \ref{lem: kk a parameter} implies that $b_{L(H),kD'}(e) \le  1- \frac{k-1}{2k^2} + \frac{(k-1)(k-2)}{6k^3\sqrt{k}} + \frac{\iota}{2}$ for every $e\in E(H)$. Thus, by Theorem \ref{thm:gct}, we have $\chi'(\cH) \le (1- \frac{k-1}{2k^2} + \frac{(k-1)(k-2)}{6k^3\sqrt{k}} + \iota)kD$, as desired.
\end{proof}

\section{A coloring procedure towards proving Theorem~\ref{thm:gct}}\label{section: 4-gctproof}

\subsection{An algorithm for sampling independent sets}\label{section: algorithm}

We use the randomized procedure introduced by Hurley, de Joannis de Verclos, and Kang \cite{HVK22} to generate a random independent set, which will be useful in assessing the behavior of color classes later in a random coloring procedure.

Define the following \texttt{algorithm}\label{algorithm}: let $\gamma>0$ be a parameter, and let $G$ be a $\Delta$-regular graph. Then generate a random independent set $I$ as follows:
\begin{enumerate}
    \item Activate each vertex of $G$ independently at random with probability $\gamma / \Delta$. Let $A$ be this set of activated vertices.
    \item For each $v\in A$, assign a ``priority,'' that is, a number $\pi(v)$ chosen independently and uniformly at random in $[0,1]$.
    \item For $v,w\in A$ such that $\{v,w\}\in E(G)$, that is, two neighboring vertices, remove the vertex with lower priority $\pi$.
\end{enumerate}
This yields the independent set $I = \{ v\in A : \pi(v) > \pi(w) \text{ for every } w\in N(v)\cap A\} $.

The goal of this subsection is to prove the following theorem concerning some properties of this independent set.

\begin{theorem} \label{thm: 2.1 - ind sets}
    For every $\iota>0$, there exists $\Delta_0$ and $\gamma_0$ such that the following holds: let $G$ be a $\Delta$-regular graph with $\Delta\ge \Delta_0$, and let $I$ be a random independent set obtained by the \textup{\texttt{algorithm}} with parameter $\gamma\ge \gamma_0$. For every vertex $r\in V(G)$, $$ \left| \p [r\in I] - \frac{1-e^{-\gamma}}{\Delta}\right| \le \frac{2}{\Delta^2}.$$
    Moreover, for every $X\subseteq N_G(r)$, we have $$ \p[I\cap X \ne \emptyset] \le \theb_{G[X\cup \{r\}], \Delta}(r) + \iota.$$
\end{theorem} 

Theorem \ref{thm: 2.1 - ind sets} is a generalization of \cite[Theorem 2.1]{HVK22}. Their result has the additional hypothesis that $G$ is \textit{$\sigma$-sparse} for some $\sigma>0$, where a graph $G$ is said to be $\sigma$-sparse if for every $v\in V(G)$, the subgraph $G[N(v)]$ has at most $(1-\sigma)\binom{\Delta(G)}{2}$ edges. They have the same conclusion that $|\Prob{r\in I}-\frac{1-e^{-\gamma}}{\Delta}| \le \frac{2}{\Delta^2}$, but their proof of this part does not require that $G$ is $\sigma$-sparse. Instead of concluding that $\p[I\cap X \ne \emptyset] \le \theb_{G[X\cup \{r\}], \Delta}(r) + \iota$ for every $X\subseteq N_G(r)$, they conclude that $\Prob{N_G(r)\cap I \ne \emptyset}\le (1-\frac{\sigma}{2}+\frac{\sigma^{3/2}}{6} + \iota)\Expect{|N_G(r)\cap I|}$, for which they do require the $\sigma$-sparsity assumption. If $G$ is $\Delta$-regular, then $G$ is $\sigma$-sparse if, and only if, $P_G(v)\ge \sigma \binom{\Delta}{2}$ for all $v\in V(G)$. By a result of Rivin \cite{RIVIN02}, for any graph $G$, we have $T_{G}(v) \le \frac{\sqrt{2}}{3}(P_G(v))^{3/2}$, whence for any $\sigma$-sparse $\Delta$-regular graph $G$, we have $b_{G,\Delta}(v) \le 1 - \frac{\sigma}{2} + \frac{\sigma^{3/2}}{6}$. 

Our proof is similar to that of \cite[Theorem 2.1]{HVK22}, and in the following proof, we reuse parts of their argument that do not require that $G$ is $\sigma$-sparse.

\begin{proof}[Proof of Theorem \ref{thm: 2.1 - ind sets}]
Let $I$ be a random independent set obtained by the \texttt{algorithm} applied to $G$ with parameter $\gamma$, and let $r\in V(G)$. The following is equivalent to \cite[Claim 1]{HVK22}; their proof does not require that $G$ is $\sigma$-sparse, so we omit the proof. 

\begin{claim} \label{clm: 1-prob} 
    For every vertex $v\in V(G)$, it holds that $$ \Prob{v\in I} = \frac{1}{\Delta} \int_0^{\gamma} (1-\frac{x}{\Delta})^{\Delta} dx \in \left[ \frac{1-e^{-\gamma}}{\Delta} - \frac{2}{\Delta^2}, \frac{1-e^{-\gamma}}{\Delta}  \right].$$
\end{claim} 
Note that Claim~\ref{clm: 1-prob} implies that $\left| \p [r\in I] - \frac{1-e^{-\gamma}}{\Delta}\right| \le \frac{2}{\Delta^2}$, as desired, so now we consider the second part of the theorem. 

As in \cite[Claim 2]{HVK22}, we may assume without loss of generality that no pair of distinct vertices in $N_G(r)$ have a common neighbor outside of $N_G[r]$, as otherwise, we can construct another $\Delta$-regular graph $G'$ where the neighborhood of $r$ is the same as in $G$ but no pair of distinct vertices in $N(r)$ have a common neighbor outside of $N[r]$ and $\Prob{I\cap X \ne \emptyset} \le \Prob{I'\cap X \ne \emptyset}$ for every $X\subseteq N(r)$, where $I'$ is the random independent set generated by the \texttt{algorithm} applied to $G'$. Now we let $X\subseteq N_G(r)$, and we proceed to show that $\p[I\cap X \ne \emptyset] \le \theb_{G[X\cup \{r\}], \Delta}(r) + \iota$.

For ease of notation, let $I_X = X\cap I$, let $\overline{P_X}:=\binom{|I_X|}{2}$,
and let $\overline{T_X}:=\binom{|I_X|}{3}$. 

Then by the Bonferroni Inequalities, 
\begin{align}
    \label{eqn Bonferroni/PIE}
    \p [I_X \ne \emptyset ] \le \E [|I_X|] - \E[\overline{P_X}] + \E [\overline{T_X}].
\end{align}

    Additionally, define $\mathcal{J}_n$ to be the collection of independent sets of size $n$ in $X$. Observe that $P_{G[X\cup \{r\}]}(r) = |\mathcal{J}_2|$ and $T_{G[X\cup \{r\}]}(r) = |\mathcal{J}_3|$.

    Given a pair $uv\in \mathcal{J}_2$, define the parameter $\ell_{uv} := \frac{1}{\Delta} |N_{G}(v) \cap N_{G}(u)|$, and note that since every common neighbor of $u$ and $v$ is in $N_G[r]\setminus \{u,v\}$, we have $\ell_{uv}\le \frac{1}{\Delta}(\Delta-1) = 1-\frac{1}{\Delta}<1$.

The following claim is proved in \cite[Claim 3 of Theorem 2.1]{HVK22} for the case of $X=N_G(r)$. The proof does not require that $G$ is $\sigma$-sparse and generalizes easily for $X\subseteq N_G(r)$, so we omit the proof. 
    
\begin{claim}\label{clm: 4.1.2 pairs expect} 
$$\E [\overline{P_X}] = \frac{1}{\Delta^2} \sum_{uv\in \mathcal{J}_2} \frac{2}{(1-\ell_{uv})(2-\ell_{uv})} + o_{\gamma}(1) + o_{\Delta}(1).$$
\end{claim}

We also need the following claim about the expected number of triples. 
\begin{claim}\label{Clm: triples} 
     \begin{align*}
        \displaystyle \E [\overline{T_X}] 
        &\le \frac{|\mathcal{J}_3|}{\Delta^3} + \frac{1}{3\Delta^2} \cdot \sum_{uv\in \mathcal{J}_2} \left( \frac{2}{(2-\ell_{uv})(1-\ell_{uv})}-1 \right).
    \end{align*}
\end{claim} 
\begin{proof}[Proof of Claim \ref{Clm: triples}]
    As shown in the proof of \cite[Claim 5 of Theorem 2.1]{HVK22} (see equation (2.4)), we have $$\Expect{\overline{T_X}} \le 
    \frac{|\mathcal{J}_3|}{\Delta^3} + \frac{1}{3\Delta^3} \cdot \sum_{uvw\in \mathcal{J}_3} \sum_{ab\in \{uv,uw,vw\} } \left( \frac{2}{(2-\ell_{ab})(1-\ell_{ab})}-1 \right),$$
    and 
    $$ \sum_{uvw\in \mathcal{J}_3} \sum_{ab\in \{uv,uw,vw\} } \left( \frac{2}{(2-\ell_{ab})(1-\ell_{ab})}-1 \right) 
    \le 
    \sum_{uv\in \mathcal{J}_2} (1-\ell_{uv})\Delta \cdot \left( \frac{2}{(2-\ell_{uv})(1-\ell_{uv})}-1 \right).$$
    Since $\ell_{uv}\ge 0$ for every $uv\in \mathcal{J}_2$, we have 
    $$ \sum_{uv\in \mathcal{J}_2} (1-\ell_{uv})\Delta \cdot \left( \frac{2}{(2-\ell_{uv})(1-\ell_{uv})}-1 \right)
    \le 
    \Delta \sum_{uv\in \mathcal{J}_2} \left( \frac{2}{(2-\ell_{uv})(1-\ell_{uv})}-1 \right).$$
    Combining the inequalities above yields the claim.
\end{proof}

\begin{claim}\label{Clm: pairs-triples} 
$$\E [\overline{P_X}] - \E[\overline{T_X}] \ge \frac{|\mathcal{J}_2|}{\Delta^2} - \frac{|\mathcal{J}_3|}{\Delta^3} + o_{\gamma}(1) + o_{\Delta}(1).$$
\end{claim}%
\begin{proof}[Proof of Claim \ref{Clm: pairs-triples}]
By Claim \ref{clm: 4.1.2 pairs expect} and \ref{Clm: triples}, we have 
\begin{align*}
    \E[\overline{P_X}] - \E[\overline{T_X}]
    & \ge  \frac{1}{\Delta^2}\sum_{uv\in \mathcal{J}_2} \frac{2}{(1-\ell_{uv})(2-\ell_{uv})} \\ & \phantom{space} - \frac{|\mathcal{J}_3|}{\Delta^3} - \frac{1}{3\Delta^2} \cdot \sum_{uv\in \mathcal{J}_2} \left( \frac{2}{(2-\ell_{uv})(1-\ell_{uv})}-1 \right)
    + o_{\gamma}(1) + o_{\Delta}(1) \\
    &= \frac{1}{3 \Delta^2} \cdot \sum_{uv\in \mathcal{J}_2} \left( \frac{4}{(2-\ell_{uv})(1-\ell_{uv})}+1\right) - \frac{|\mathcal{J}_3|}{\Delta^3} + o_{\gamma}(1) + o_{\Delta}(1).
\end{align*}
Since the function $x\mapsto \frac{4}{(2-x)(1-x)}$ is increasing on $[0,1)$, we have
$$ \frac{1}{3\Delta^2} \cdot \sum_{uv\in \mathcal{J}_2} \left( \frac{4}{(2-\ell_{uv})(1-\ell_{uv})} +1\right) \ge \frac{1}{3\Delta^2} \cdot \sum_{uv\in \mathcal{J}_2} \left( \frac{4}{(2)(1)}+1 \right) = \frac{|\mathcal{J}_2|}{\Delta^2}.$$
Combining the inequalities above yields
$$ \E [\overline{P_X}] - \E[\overline{T_X}] \ge \frac{|\mathcal{J}_2|}{\Delta^2} - \frac{|\mathcal{J}_3|}{\Delta^3} + o_{\gamma}(1) + o_{\Delta}(1),$$
as claimed.
\end{proof}

Since $|\mathcal{J}_2| = P_{G[X\cup \{r\}]}(r)$ and $|\mathcal{J}_3| = T_{G[X\cup \{r\}]}(r)$, we have $ \frac{|\mathcal{J}_2|}{\Delta^2} - \frac{|\mathcal{J}_3|}{\Delta^3} =\frac{|X|}{\Delta}-  \theb_{G[X\cup \{r\}],\Delta}(r)$ by definition. 

By Claim \ref{clm: 1-prob}, we have $\Expect{|I_X|} \le |X|/\Delta + o_{\Delta}(1)$. Hence, by Claim \ref{Clm: pairs-triples} and \eqref{eqn Bonferroni/PIE}, 
\begin{align*}
    \p[I_X \ne \emptyset] &\le \Expect{|I_X|} - \Expect{\overline{P_X}} + \Expect{\overline{T_X}} \\
    &\le \frac{|X|}{\Delta} - \frac{|\mathcal{J}_2|}{\Delta^2} + \frac{|\mathcal{J}_3|}{\Delta^3} + o_{\gamma}(1) + o_{\Delta}(1) \\
    &= \theb_{G[X\cup \{r\}], \Delta}(r) + o_{\gamma}(1) + o_{\Delta}(1),
\end{align*}
as desired. 
\end{proof}

\subsection{One nibble to partially color}
Let $\cH$ be a hypergraph. A {\it list assignment} $L$ for $E(\cH)$ is a collection of lists of colors $L(e)$, and  a \textit{partial $L$-coloring} of $E(\cH)$ is an assignment $c:E_1 \to \bigcup L(e)$ where $E_1\subseteq E(\cH)$ such that $c(e) \in L(e)$ for all $e\in E_1$ and $c(e_1) \neq c(e_2)$ for every pair of distinct edges $e_1,e_2 \in E_1$ such that $e_1 \cap e_2 \neq \emptyset$. Note that $c$ is a proper edge-coloring of the hypergraph induced by $\cH$ on $E_1$. We call it a partial coloring because there may remain some uncolored edges. The coloring is \textit{proper} if $c(e_1) \ne c(e_2)$ for every pair of distinct edges $e_1,e_2\in E_1$ where $e_1\cap e_2 \ne \emptyset$.
Given a partial $L$-coloring $c:E_1 \to \bigcup L(e)$, we say the \textit{residual hypergraph} $\cH_c$ is the hypergraph obtained from $\cH$ by deleting all the edges in $E_1$, that is, $\cH_c$ is the subgraph of $\cH$  induced by all uncolored edges.  The \textit{residual list assignment} $L_c$ is a list assignment for $E(\cH_c)$ such that for $e\in E(\cH_c)$, the list of colors $L_c(e)$ for $e$ is the list obtained from $L(e)$ by deleting color $\alpha$ if and only if there exists some $f\in E_1$ such that $c(f) = \alpha$ and $f\cap e \ne \emptyset$. That is, edge $e$ keeps a particular color in its list if there is no neighboring edge that has been assigned that color.

In our proof of Theorem~\ref{thm:gct}, we will apply the following lemma iteratively to construct a partial proper edge-coloring of $\cH$. In each application of the lemma, we find a partial $L$-coloring for some list assignment $L$ for $E(\cH)$ in which the residual hypergraph has a smaller maximum
degree with residual list sizes that are not too much smaller. The proof is similar to that of \cite[Lemma 2.2]{HVK22}; we consider a random partial coloring and apply the Lovász Local Lemma, but we have modified the coloring procedure in a way that makes the analysis slightly simpler. In
particular, by generating each color class independently, we can use McDiarmid’s Inequality to show concentration of degrees in the residual hypergraph and the Chernoff Bounds to show concentration of the residual list sizes instead of the more complicated Talagrand’s Inequality.

\begin{lemma} \label{thm: 2.2 - nibble}
    For every $A, k \in \mathbb{N}$ and $\iota > 0$, there exists $\gamma_0 > 0$ such that for all $\gamma \ge \gamma_0$, the following holds for all sufficiently large $D$. Let $\mathcal{H}$ be a $k$-uniform hypergraph with $\Delta(\mathcal{H}) \le D$, and let $L : E(\mathcal{H}) \to 2^{[AD]}$ be a list assignment for $E(\mathcal{H})$ where each list has size at least $\ell := \lceil kD/\gamma \rceil$. If  $b \ge \max_{e \in E(\mathcal{H})} b_{L(\mathcal{H}), kD}(e)$, then there is a partial $L$-coloring 
$c$ of $E(\mathcal{H})$ such that the residual hypergraph  $\mathcal{H}_c$ and residual list assignment $L_c$  satisfy
    \begin{enumerate}[label=\textup{(N\arabic*)}]
        \item \label{thm: 2.2 prop N1}  $\Delta(\cH_c) \le D' := (1-\frac{1-\iota}{\gamma})D$, and
        \item \label{thm: 2.2 prop N2} $|L_c(e)|\ge \ell':=(1 -b-\iota) \ell$ for every $e\in E(\cH_c)$.
    \end{enumerate}
\end{lemma}

\begin{proof}
    
Let $G := L(\mathcal{H})$ be the line graph of $\mathcal{H}$, let $\Delta := kD$, and note that $\Delta(G) \le \Delta$. We assume without loss of generality that $|L(v)| = \ell$ for each $v \in V(G)$, as otherwise we can remove $|L(v)| - \ell$ colors from a vertex $v$'s list.

We will consider each color independently. For each color $\alpha$, define $G_\alpha$ to be the subgraph of $G$ induced by the vertices $v$ with $\alpha \in L(v)$. Using a standard argument (see \cite[Section 1.5]{molloy2002txt}), we can embed each $G_\alpha$ in a $\Delta$-regular graph $G'_\alpha$ which contains $G_\alpha$ as an induced subgraph.

Now, for each color $\alpha$ which appears in any vertex's list, apply the \texttt{algorithm} defined in Section \ref{section: algorithm} with parameter $\gamma$ to generate a random independent set $I(\alpha)$ in $G'_\alpha$. We construct a random partial $L$-coloring $c : E(\mathcal{H}) \cap \bigcup I(\alpha) \to \bigcup L(e)$ of $E(\mathcal{H})$ as follows. If $e \in E(\mathcal{H}) \cap I(\alpha)$, then let $c(e) = \alpha$, breaking ties arbitrarily should $e$ appear in multiple independent sets. By construction, $c$ is a partial $L$-coloring of $E(\mathcal{H})$, and we show that with non-zero probability, the residual hypergraph $\mathcal{H}_c$ and residual list assignment $L_c$ satisfy \ref{thm: 2.2 prop N1} and \ref{thm: 2.2 prop N2}.

Let $G_c$ denote the line graph of the residual hypergraph $\mathcal{H}_c$. For each vertex $v \in V(G)$, define the random variable
\[ Del_c(v) := \{ \alpha \in L(v) : I(\alpha) \cap N_G(v) \neq \emptyset \}, \]
and note that for every $v \in V(G_c)$, we have $|L_c(v)| \ge \ell - |Del_c(v)|$. To simplify the probabilistic analysis, we examine $|Del_c(v)|$ rather than $|L_c(v)|$. To show that $\mathcal{H}_c$ and $L_c$ satisfy \ref{thm: 2.2 prop N1} and \ref{thm: 2.2 prop N2} with non-zero probability, we use the Lovász Local Lemma. To that end, we define the following sets of ``bad'' events:
\begin{itemize}
    \item for each $x \in V(\mathcal{H})$, let $B_1(x)$ be the event that $\text{deg}_{\mathcal{H}_c}(x) > D'$, and
    \item for each $v \in V(G)$, let $B_2(v)$ be the event that $|Del_c(v)| > \ell - \ell'$.
\end{itemize}

We will show that the probability of each of these bad events is small. First, we compute the expected degree of a vertex $x \in V(\mathcal{H})$ in $\mathcal{H}_c$ and the expected size of $Del_c(v)$ for a vertex $v \in V(G)$, and then we show these random variables are concentrated around their expectation with high probability.
\begin{claim}\label{claim4.2.1}
    For every $x \in V(\mathcal{H})$,
\[ \mathbb{E} \left[ \text{deg}_{\mathcal{H}_c}(x) \right] \le D' - \frac{\iota}{2\gamma} \cdot D. \]
\end{claim}

\begin{proof}
By the linearity of expectation, we have
\[ \mathbb{E} \left[ \text{deg}_{\mathcal{H}_c}(x) \right] = \sum_{e \in E(\mathcal{H}): e \ni x} \mathbb{P} [e \in E(\mathcal{H}_c)], \] 
and for each $e \in E(\mathcal{H})$, we have
\[ \mathbb{P} [e \in E(\mathcal{H}_c)] = \mathbb{P} \left[ e \notin \bigcup I(\alpha) \right] = \prod_{\alpha \in L(e)} \mathbb{P} [e \notin I(\alpha)], \]
where the last equality uses that the independent sets $I_\alpha$ are generated independently. By Theorem ~\ref{thm: 2.1 - ind sets}, for every $e \in E(\mathcal{H})$ and $\alpha \in L(e)$, we have
\[ \mathbb{P} [e \notin I(\alpha)] \le 1 - \frac{1 - e^{-\gamma}}{\Delta} + \frac{2}{\Delta^2}, \]
so by the inequalities above, we have
\[ \mathbb{E} [\text{deg}_{\mathcal{H}_c}(x)] \le \sum_{e \in E(\mathcal{H}): e \ni x} \prod_{\alpha \in L(e)} \left( 1 - \frac{1 - e^{-\gamma}}{\Delta} + \frac{2}{\Delta^2} \right) \le D \left( 1 - \frac{1 - e^{-\gamma}}{\Delta} + \frac{2}{\Delta^2} \right)^\ell. \]
Using the inequalities $1 + a \le e^a \le 1 + a + a^2$ for $a \in (-1, 1)$ and that $\ell \ge \Delta/\gamma$, we have
\[ \left( 1 - \frac{1 - e^{-\gamma}}{\Delta} + \frac{2}{\Delta^2} \right)^\ell \le e^{-\frac{1 - e^{-\gamma}}{\gamma} + \frac{2}{\Delta\gamma}} \le 1 - \frac{1 - e^{-\gamma}}{\gamma} + \frac{2}{\Delta\gamma} + \left( -\frac{1 - e^{-\gamma}}{\gamma} + \frac{2}{\Delta\gamma} \right)^2. \]
For $\gamma$ and $\Delta$ sufficiently large, we have
\[ \frac{e^{-\gamma}}{\gamma} + \frac{2}{\Delta\gamma} + \left( -\frac{1 - e^{-\gamma}}{\gamma} + \frac{2}{\Delta\gamma} \right)^2 \le \frac{\iota}{2\gamma}. \]
Combining the three inequalities above yields the claim, since $D' = (1 - 1/\gamma)D + \iota D/\gamma$.
\end{proof} 
\begin{claim}\label{claim4.2.2}
 For every $v \in V(G)$,
\[ \mathbb{E} [|Del_c(v)|] \le \left( b + \frac{\iota}{2} \right) \ell. \]   
\end{claim}

\begin{proof} By the linearity of expectation, we have
\[ \mathbb{E} [|Del_c(v)|] = \sum_{\alpha \in L(v)} \mathbb{P} [I(\alpha) \cap N_G(v) \neq \emptyset], \]
and for every $\alpha \in L(v)$, we have $I(\alpha)\cap N_G(v)=I(\alpha)\cap N_{G_\alpha}(v)$. Hence, by Theorem~\ref{thm: 2.1 - ind sets} with $G'_\alpha$, $v$, $N_{G_\alpha}(v)$, and $\iota/2$ playing the roles of $G$, $r$, $X$, and $\iota$, respectively, we have
\[ \mathbb{P} [I(\alpha) \cap N_G(v) \neq \emptyset] \le b_{G'_\alpha[N_{G_\alpha}[v]], \Delta}(v) + \frac{\iota}{2} = b_{G_\alpha, \Delta}(v) + \frac{\iota}{2}. \]
Thus, since $|L(v)| = \ell$, it suffices to show that for every $\alpha \in L(v)$, we have $b_{G_\alpha, \Delta}(v) \le b$. Indeed, let $d := \text{deg}_G(v) - \text{deg}_{G_\alpha}(v)$, and note that $P_{G_\alpha}(v) \ge P_G(v) - d\Delta$ and $T_{G_\alpha}(v) \le T_G(v)$, whence
\[ b_{G_\alpha, \Delta}(v) \le \frac{\text{deg}_G(v) - d}{\Delta} - \frac{P_G(v) - d\Delta}{\Delta^2} + \frac{T_G(v)}{\Delta^3} = b_{G, \Delta}(v) \le b, \]
as desired.
\end{proof}

\begin{claim}\label{cliam2.4.3}
  For every $x \in V(\mathcal{H})$,
\[ \mathbb{P} \left[ \text{deg}_{\mathcal{H}_c}(x) > D' \right] \le D^{-\log D}. \]  
\end{claim}

\begin{proof} The random variable $\text{deg}_{\mathcal{H}_c}(x)$ is determined by the independently constructed sets $I(\alpha)$ for each color $\alpha$. Notice that if $I(\alpha)$ is changed for a single color $\alpha$, then this changes $\text{deg}_{\mathcal{H}_c}(x)$ by at most one, since at most one edge containing $x$ can be assigned color $\alpha$. That is, the function $\text{deg}_{\mathcal{H}_c}(x)$ is 1-Lipschitz. Hence, McDiarmid's Inequality \cite{McDiarmid_1989} implies that
\[ \mathbb{P} \left[ \text{deg}_{\mathcal{H}_c}(x) \ge \mathbb{E} \left[ \text{deg}_{\mathcal{H}_c}(x) \right] + t \right] \le \exp \left( -\frac{2t^2}{\sum_{\alpha}(1^2)} \right) \]
for any $t > 0$. We apply this result with $t = \iota D/(2\gamma)$. By assumption, the total number of colors is at most $AD$, so $\sum_{\alpha}(1^2) = AD$. Thus, by Claim \ref{claim4.2.1}, we have
\[ \mathbb{P} \left[ \text{deg}_{\mathcal{H}_c}(x) > D' \right] \le \exp \left( -\frac{2((\iota D)/(2\gamma))^2}{AD} \right) = \exp \left( -\frac{\iota^2}{2\gamma^2 A} D \right) \le D^{-\log D}, \]
as desired. \end{proof}
\begin{claim}\label{cliam2.4.4}
 For every $v \in V(G)$,
\[ \mathbb{P} \left[ |Del_c(v)| \ge \ell - \ell' \right] \le D^{-\log D}. \]   
\end{claim} 

\begin{proof} The random variable $|Del_c(v)|$ is the sum of independent Bernoulli random variables, since each color class is generated independently. By the Chernoff Bounds \cite{Chernoff}, we have
\[ \mathbb{P} \left[ |Del_c(v)| \ge \mathbb{E} \left[ |Del_c(v)| \right] + t \right] \le \exp \left( -\frac{t^2}{2\mathbb{E} \left[ |Del_c(v)| \right] + t} \right) \]
for any $t \ge 0$. We apply this result with $t = \ell \cdot \frac{\iota}{2}$. By Claim \ref{claim4.2.2}, we have
\[ t = \ell \cdot \frac{\iota}{2} = \ell - (1 - b-\iota)\ell - \left( b + \frac{\iota}{2} \right)\ell \le \ell - \ell' - \mathbb{E} \left[ |Del_c(v)| \right], \]
and thus,
\begin{align*}
\mathbb{P} \left[ |Del_c(v)| \ge \ell - \ell' \right] 
& \le \mathbb{P} \left[ |Del_c(v)| \ge \Expect{|Del_c(v)|} + t \right] \\
&\le \exp \left( -\frac{(\ell \cdot \iota /2)^2}{2\ell(b + \iota/2) + \ell \cdot \iota /2} \right) \\
&\le \exp \left( -\frac{\iota^2}{16}\ell \right) \le D^{-\log D},
\end{align*}
where the last inequality uses that $\ell \ge kD/\gamma$ and $D$ is sufficiently large.
\end{proof}

Let $\mathcal{B} := \{B_1(x) : x \in V(\mathcal{H})\} \cup \{B_2(v) : v \in V(G)\}$ be the set of all bad events we wish to prevent. Note that each color class $I_\alpha$ generated by the \texttt{algorithm} defined in Section \ref{section: algorithm} is determined by a set of $2|V(G)|$ independent trials: one trial for each vertex $v$ determining whether it is ``activated'', and another trial to determine $v$'s priority. Moreover, for every $x \in V(\mathcal{H})$, the event $B_1(x)$ is determined by the trials for edges with distance at most 1 from an edge containing $x$ in $\mathcal{H}$, and for every $v \in V(G)$, the event $B_2(v)$ is determined by the trials for vertices with distance at most 2 from $v$ in $G$. Hence, each event in $\mathcal{B}$ is mutually independent of all but at most $2\Delta^4$ other events in $\mathcal{B}$. Therefore, by Claims \ref{cliam2.4.3} and \ref{cliam2.4.4} and the Lovász Local Lemma, with positive probability, neither of the bad events $B_1(x)$ nor $B_2(v)$ happens. For such a coloring, the residual hypergraph and list assignment satisfy \ref{thm: 2.2 prop N1} and \ref{thm: 2.2 prop N2}, which completes the proof.
\end{proof} 
\subsection{Proof of Theorem~\ref{thm:gct}}

The result in the previous section, Lemma \ref{thm: 2.2 - nibble}, consisted of one nibble step of a coloring procedure. We need only iterate this, and show that our random coloring procedure pans out in context.

\begin{proof}[Proof of Theorem~\ref{thm:gct}]

  It suffices to prove the result for $0<\eps\le 1$, since every $k$-uniform hypergraph of maximum degree at most $D$ has chromatic index at most $kD$.
  Given $\eps$ and $k$, set $A:=\left\lceil 3(1+\eps)k/\eps\right\rceil$. We choose $\gamma$ sufficiently large with respect to $A$ and $\eps$, and then let $D$ be sufficiently large with respect to $\gamma$. Let $\mathcal{H}$ be a given $k$-uniform hypergraph satisfying the hypothesis of the statement. We will iteratively color the vertices of $L(\mathcal{H})$ by repeatedly applying Lemma \ref{thm: 2.2 - nibble}.

For $i\ge 1$, let
\begin{align*}
    D_i&:= \left(1-\frac{1-\eps/6}{\gamma}\right)^{i-1}D, \\
    \ell_i&:= \frac{kD_i}{\gamma},\\
    \ell_i'&:=(1-\theb - \eps/6) \ell_i, \\
    K_i &:= \left\lceil(\theb+(\eps/2))k(D-D_i) + \frac{kD}{\gamma}\right\rceil,
\end{align*}
and let $L^{(i)}$ be the list assignment for $E(\mathcal{H})$ where $L^{(i)}(e) = [K_i]$ for every edge $e \in \mathcal{H}$.

\begin{claim}\label{clm: 4.2.5ell}
    For all $i\ge 1$ such that $D_i\ge \eps D/3$, $$ \ell'_i + K_{i+1} - K_i \ge \ell_{i+1}.$$
\end{claim}

\begin{proof}
    We have
    \begin{align}
\ell'_i + K_{i+1} - K_i
&\ge (1 - \theb - \eps/6) \ell_i + \left( \theb + \frac{\eps}{2} \right) k(D_i - D_{i+1})-1 \nonumber \\
&= (1 - \theb - \eps/6) \frac{kD_i}{\gamma} + \left( \theb + \frac{\eps}{2} \right) k \left( \frac{1 - \eps/6}{\gamma} \right) D_i-1
\nonumber\\
&= \frac{kD_i}{\gamma} \left(1+\frac{\eps}{2} - \frac{\eps}{6}(1+\theb + \eps/2)\right)-1 \nonumber \\
&\ge \frac{kD_i}{\gamma} (1+\eps/12)-1 \ge \frac{kD_{i+1}}{\gamma} = \ell_{i+1}. \nonumber \qedhere
\end{align}
\end{proof}

\begin{claim}\label{clm: gct claim}
For all $i \ge 1$ such that either $i=1$ or $D_{i-1} \ge \eps D/3$, there exists a partial $L^{(i)}$-coloring $c_i$ of $E(\mathcal{H})$ such that the residual hypergraph $\mathcal{H}_{c_i}$ and residual list assignment $L_{c_i}^{(i)}$ satisfy
\begin{itemize}
    \item $\Delta(\mathcal{H}_{c_i}) \le D_i$ and
    \item $|L_{c_i}^{(i)}(e)| \ge \ell_i$ for every edge $e \in E(\mathcal{H}_{c_i})$.
\end{itemize}
\end{claim}

\begin{proof}[Proof of Claim~\ref{clm: gct claim}]
    We proceed by induction. As a base case, for $i=1$, the trivial coloring $c_1 : \emptyset \to \emptyset$, regarded as a partial $L^{(1)}$-coloring, satisfies both conditions.

   Now suppose $i\ge 1$ and $D_i \ge \eps D/3$. We will prove the claim for $i+1$. By the inductive hypothesis, there is a partial $L^{(i)}$-coloring $c_i$ of $E(\mathcal{H})$ such that the residual hypergraph $\mathcal{H}_i := \mathcal{H}_{c_i}$ and residual list assignment $L_i := L_{c_i}^{(i)}$ satisfy $\Delta(\mathcal{H}_i) \le D_i$ and $|L_i(e)| \ge \ell_i$ for every edge $e \in E(\mathcal{H}_i)$.

Since $D_i \ge \eps D/3$, our choices of $\gamma$ and $D$ ensure that $K_i\le (1+\eps)kD\le AD_i$, and we have $\theb \ge \max_{e \in E(\mathcal{H}_i)} b_{L(\mathcal{H}_i), kD_i}(e)$. Thus, by Lemma~\ref{thm: 2.2 - nibble} applied to $\mathcal{H}_i$ with $D_i$, $\eps/6$, and $A$ playing the roles of $D, \iota$, and $A$, respectively, there is a partial $L_i$-coloring $c_i'$ of $E(\mathcal{H}_i)$ such that the resulting residual hypergraph has maximum degree at most $\left( 1 - \frac{1 - \eps/6}{\gamma} \right) D_i = D_{i+1}$ and the residual list sizes are at least $(1 - \theb - \eps/6)\ell_i = \ell'_i$. Combining $c_i$ and $c_i'$ yields a partial $L^{(i+1)}$-coloring $c_{i+1}$. Claim~\ref{clm: 4.2.5ell} gives $|L_{c_{i+1}}^{(i+1)}(e)| \ge \ell'_i + K_{i+1} - K_i \ge \ell_{i+1}$ for every edge $e \in E(\mathcal{H}_{c_{i+1}})$. Thus $c_{i+1}$ satisfies both conditions.
\end{proof}

Let $i\ge 2$ be maximum such that $D_{i-1}\ge \eps D/3$. By Claim~\ref{clm: gct claim}, there is a partial $L^{(i)}$-coloring $c_i$ satisfying $\Delta(\mathcal{H}_{c_i})\le D_i<\eps D/3$.

We have $K_i \le (\theb+\eps/2)kD + \frac{kD}{\gamma}+1  \le (\theb+\eps/2 + \frac{1}{\gamma})kD + 1$. The residual hypergraph $\cH_{c_i}$ satisfies $\Delta(\cH_{c_i}) < \eps D/3$, whence we can finish coloring the edges with at most $1+\eps kD /3$ additional colors. Altogether, by our choices of $\gamma$ and $D$, we color the edges of $\cH$ with at most $(\theb+\eps ) kD$ colors.
\end{proof}


\providecommand{\bysame}{\leavevmode\hbox to3em{\hrulefill}\thinspace}
\providecommand{\MR}{\relax\ifhmode\unskip\space\fi MR }
\providecommand{\MRhref}[2]{%
  \href{http://www.ams.org/mathscinet-getitem?mr=#1}{#2}
}
\providecommand{\href}[2]{#2}

\end{document}